\documentclass[a4paper,11pt]{amsart}
\usepackage{amsaddr}
\usepackage{amsfonts}
\usepackage{amssymb}
\usepackage[utf8]{inputenc}
\usepackage{amsmath}
\usepackage{pdflscape}
\usepackage{graphicx}
\usepackage[dvipsnames]{xcolor}
\usepackage{enumerate}
\usepackage[colorlinks=true, linkcolor=red, citecolor=red]{hyperref}

\newtheorem{Theorem}{Theorem}[section]
\newtheorem{Definition}[Theorem]{Definition}
\newtheorem{Proposition}[Theorem]{Proposition}
\newtheorem{Lemma}[Theorem]{Lemma}
\newtheorem{Corollary}[Theorem]{Corollary}

\theoremstyle{remark}
\newtheorem{Remark}[Theorem]{Remark}
\newtheorem{Remarks}[Theorem]{Remarks}

\newcommand{\RR}{\mathbb{R}}
\newcommand{\px}{\partial_x}
\newcommand{\pt}{\partial_t}

\providecommand{\norm}[1]{\left\lVert#1\right\rVert}
\newcommand{\qtq}[1]{\quad \text{#1}\quad }
\newcommand{\R}{{\mathbb R}}
\newcommand{\T}{{\mathbb T}}

\newcommand{\N}{{\mathbb N}}

\newcommand{\B}{{\mathcal{B}}}
\newcommand{\A}{\mathcal{A}}

\providecommand{\norm}[1]{\left\lVert#1\right\rVert}
\numberwithin{equation}{section}
\makeatletter
\@namedef{subjclassname@2010}{\textup{2020} Mathematics Subject Classification}
\makeatother

\begin{document}

\title[Time--delayed KdV--type systems: A detailed study]{A qualitative study of the generalized dispersive systems with time--delay: The unbounded case}

\author[Capistrano--Filho]{Roberto de A. Capistrano--Filho*}
\thanks{*Corresponding author.}
\thanks{\today}
\address{Departamento de Matem\'atica, Universidade Federal de Pernambuco, \\S/N Cidade Universit\'aria, 50740-545, Recife (PE), Brazil\\
Email address: \normalfont\texttt{roberto.capistranofilho@ufpe.br}}

\author[Gallego]{Fernando A. Gallego}
\address{Departamento de Matem\'atica y Estadística, Universidad Nacional de Colombia (UNAL), \\Cra 27 No. 64-60, 170003, Manizales, Colombia\\
Email address: \normalfont\texttt{fagallegor@unal.edu.co}}

\author[Komornik]{Vilmos Komornik}
\address{Département de Mathématique,
Université de Strasbourg,
\\7 rue René Descartes,
67084 Strasbourg Cedex, France\\
Email address: \normalfont\texttt{komornik@math.unistra.fr}}


\numberwithin{equation}{section}

\begin{abstract}
We study the asymptotic behavior of the solutions of the time-delayed higher-order dispersive nonlinear differential equation 
\begin{equation*}
u_t(x,t)+Au(x,t) +\lambda_0(x) u(x,t)+\lambda(x) u(x,t-\tau )=0
\end{equation*}
where
\begin{equation*}
Au=(-1)^{j+1}\partial_x^{2j+1}u+(-1)^m\partial_x^{2m}u+ \frac{1}{p+1}\partial_xu^{p+1}
\end{equation*}
with $m\le j$ and $1\le p<2j$.
Under suitable assumptions on the time delay coefficients, we prove that the system is exponentially stable if the coefficient of the delay term is bounded from below by a suitable positive constant, without any assumption on the sign of the coefficient of the undelayed feedback.  
Additionally, in the absence of delay, general results of stabilization are established in $H^s(\mathbb{R})$ for $s\in[0,2j+1]$. 
Our results generalize several previous theorems for the Korteweg-de Vries type delayed systems in the literature. 
\end{abstract}
\maketitle


\section{Introduction}
\subsection{Model description}
Under suitable assumptions on amplitude, wavelength, wave steepness, and so on, the study on asymptotic models for water waves has been extensively investigated to understand the full water wave system; see, for instance, \cite{ASL,BCL, BLS} and references therein for a rigorous justification of various asymptotic models for surface and internal waves.

Formulating the waves as a free boundary problem of the incompressible,  irrotational Euler equation in an appropriate non-dimensional form, one has two non-dimensional parameters $\delta := \frac{h}{\lambda}$ and $\varepsilon := \frac{a}{h}$, where the water depth, the wavelength and the amplitude of the free surface are parameterized as $h, \lambda$ and $a$, respectively. Moreover, another non-dimensional parameter $\mu$ is called the Bond number, which measures the importance of gravitational forces compared to surface tension forces. The physical condition $\delta \ll 1$ characterizes the waves, called long waves or shallow water waves.  In particular, considering the relations between $\varepsilon$ and $\delta$, we can have two well-known regimes:
\begin{enumerate}
\item[1.] Korteweg-de Vries (KdV): $\varepsilon = \delta^2 \ll 1$ and $\mu \neq \frac13$.
Under this regime, Korteweg and de Vries \cite{Korteweg}\footnote{This equation was first introduced by Boussinesq \cite{Boussinesq}, and Korteweg and de Vries rediscovered it twenty years later.} derived the following  well-known equation as a central equation among other dispersive or shallow water wave models called the KdV equation from the equations for capillary-gravity waves: 
\[\pm2 \eta_t + 3\eta\eta_x +\left( \frac13 - \mu\right)\eta_{xxx} = 0.\]
\item[2.] Kawahara: $\varepsilon = \delta^4 \ll 1$ and $\mu = \frac13 + \nu\varepsilon^{\frac12}$.
In connection with the  critical Bond number $\mu = \frac13$, Hasimoto \cite{Hasimoto1970} derived a fifth-order KdV equation of the form
\[\pm2 \eta_t + 3\eta\eta_x - \nu \eta_{xxx} +\frac{1}{45}\eta_{xxxxx} = 0,\]
which is nowadays called the Kawahara equation.
\end{enumerate}

In recent years, many authors have been interested in finding the behavior of solutions for the time-delayed KdV equation, time-delayed Kawahara equation, and other time-delayed dispersive systems. See, for instance, \cite{Chentouf,Chentouf1,CaVi,CCKR,KomPig2020, Valein} and the reference therein.  In this article, our goal is to study general results for a higher-order dispersive equation, which in some sense recovers the equations mentioned in the cited articles.  Due to this advance for this type of dispersive equation, our main focus is to investigate the stabilization of the higher-order extension, for example, of KdV and Kawahara equations.

To be precise, we will consider the Cauchy problem for the following higher-order KdV-type equation posed in $\mathbb{R}$:
\begin{equation}\label{eq:kdv_int}
\begin{cases}
u_t(x,t)+(-1)^{j+1}\partial_x^{2j+1}u(x,t) +\frac12 \partial_x(u^2)=0, & (t,x) \in \mathbb{R} \times \mathbb{R},\\
u(0,x) = u_0(x), & x\in \mathbb{R}.
\end{cases}
\end{equation}
Specifically, \eqref{eq:kdv_int} is called KdV and fifth-order KdV--type equation when $j=1$ and $j=2$, respectively.   More generally, we aim to prove the stabilization results of the solutions for a time-delayed higher-order dispersive system with a strong dissipative term, namely
\begin{equation}\label{11}
\begin{cases}
u_t(x,t)+(-1)^{j+1}\partial_x^{2j+1}u(x,t)+(-1)^m\partial_x^{2m}u(x,t) +\lambda_0(x) u(x,t)\\\hspace{4cm}+\lambda(x) u(x,t-\tau )+{\frac{1}{p+1}\partial_xu^{p+1}(x,t)}=0,
&\qtq{in}\RR\times (0,\infty),\\
u(x, s)=u_0(x, s)\,&\qtq{in}\RR\times [-\tau, 0],
\end{cases}
\end{equation}
where we often write $u_0(x)$ instead of $u_0(x,0)$ with $u_0$ the initial value\footnote{Here, we consider the following compatibility condition $u_0(x,0)=u_0(x)$.} and $j,m\in\mathbb{N}$,  $m\leq j$, $1\leq p < 2j$.  Here,  the constant $\tau >0$ is the time delay and the coefficients are considered with the following regularity $$\lambda_0(x), \lambda(x) \in L^\infty(\RR).$$  Thus, our main intention is to furnish sufficient conditions on the coefficients $\lambda, \lambda_0$ to have well-posedness and exponential decay estimates for the model \eqref{11}.

 It is crucial to emphasize that the initial condition $u(x, s) = u_0(x, s)$  must be taken into account, as it arises from the initial-boundary value problem for a transport equation. This is a classical setup, as detailed in \cite{ NPSicon06}. Specifically, the well-posedness of the delay differential equation \eqref{11} is studied using the following change of variable: $z(x, \rho, t) = u(x, t - \rho \tau),$ where  $x \in \mathbb{R}$,  $\rho \in (0, 1)$, with  $t > 0$. Substituting this, the function  $z(x, \rho, t)$  satisfies the transport equation:
\begin{equation}\label{eqtransport}
\begin{cases}
\tau \partial_t z(x, \rho, t) + \partial_\rho z(x, \rho, t) = 0, & x \in \mathbb{R}, \rho \in (0, 1), t > 0, \\
z(x, 0, t) = u(x, t), & x \in \mathbb{R}, t > 0, \\
z(x, \rho, 0) = z_0(x, \rho, -\rho \tau), & x \in \mathbb{R}, \rho \in (0, 1),
\end{cases}
\end{equation}
where  $z_0(x, \rho, -\rho \tau) := u_0(x, -\rho \tau)$. With this formulation, the linear system associated with \eqref{11} can be rewritten as an abstract problem. Let  $U(t) = (u(\cdot, t), z(\cdot, \cdot, t))$ and denote  $z(1) := z(x, 1, t)$. Incorporating the linear system derived from \eqref{11} and \eqref{eqtransport}, we obtain the following system:
\begin{equation*}
\begin{cases}
u_t + (-1)^{j+1} \partial_x^{2j+1} u + (-1)^m \partial_x^{2m} u + \lambda_0(x) u + \lambda(x) z(1) = 0, & (x, t) \in \mathbb{R} \times (0, \infty), \\
\tau \partial_t z(\rho) + \partial_\rho z(\rho) = 0, & (x, t) \in \mathbb{R} \times (0, \infty),  \rho \in (0, 1), \\
z(x, 0, t) = u(x, t), & (x, t) \in \mathbb{R} \times (0, \infty), \\
u(x, 0) = u_0(x), & x \in \mathbb{R}, \\
z(x, \rho, 0) = z_0(x, \rho, -\rho \tau), & x \in \mathbb{R}, \rho \in (0, 1),
\end{cases}
\end{equation*}
which is equivalent to the abstract Cauchy problem:
\[
\begin{cases}
\frac{d}{d t} U(t) = A U(t), \\
U(0) = \left(u_0(x), z_0(x, -\rho \tau)\right),
\end{cases}
\]
where  $A: D(A) \subset \mathcal{H} \to \mathcal{H}$  is defined as
$$A(u, z) = \left((-1)^j \partial_x^{2j+1} u - (-1)^m \partial_x^{2m} u - \lambda_0(x) u - \lambda(x) z(1), -\tau^{-1} \partial_\rho z\right),$$
and the dense domain  $D(A) = \left\{ U \in \mathcal{H} : A U \in \mathcal{H} \right\}$, where  $\mathcal{H}$ is an appropriately chosen Hilbert space. Henceforth, we will focus on analyzing the stabilization properties of the system \eqref{11}.

\subsection{Review of the results in the literature}
Let us briefly discuss the preceding works \cite{Chentouf,Chentouf1,CaVi,CCKR,KomPig2020, Valein} and the results concerning the well-posedness of \eqref{eq:kdv_int}.

The local and global well-posedness of \eqref{eq:kdv_int} has been widely studied. 
The local well-posedness result was first proved by Gorsky and Himonas \cite{Gorsky:2009eg} for $s\ge-\frac12$, and Hirayama \cite{Hirayama} improved this to $s \ge -\frac{j}{2}$. 
Both works were based on the standard Fourier restriction norm method. 
Hirayama improved the bilinear estimate by using the factorization of the resonant function. The global well-posedness of \eqref{eq:kdv_int} was established for $j=1,2$ by Colliander et al. \cite{CKSTT1} and Kato \cite{Kato}, respectively, via the ``I-method". 
In \cite{HongKwak}, the authors extended the results of \cite{CKSTT1} and \cite{Kato} to $j\ge3$. 
Their method follows the argument in \cite{CKSTT1} for the periodic KdV equation, while some estimates are slightly different. 
They proved that the IVP \eqref{eq:kdv_int} is globally well-posed in $H^s(\T)$ for $j \ge 3$ and $s \ge -\frac{j}{2}$.

To our knowledge, the only work concerning the control and stabilization properties for the system \eqref{eq:kdv_int} was done by the first author with two collaborators in \cite{CaKa}. 
The authors studied the local and global control results for \eqref{eq:kdv_int} posed on the unit circle.  
More precisely,  they considered the system 
\begin{equation*}
\begin{cases}
\pt u + (-1)^{j+1} \px^{2j+1}u + \frac12 \px(u^2)=f(t,x), & (t,x) \in \R \times \T,\\
u(0,x) = u_0(x), &x\in H^s(\T),
\end{cases}
\end{equation*}
posed on a periodic domain $\mathbb{T}$.
They showed the globally controllability in $H^s(\T)$, for $s\geq0$ by the forcing term
\begin{equation*}
f(t,x):=g(x)  \left(  h(t,x)  -\int_{\mathbb{T}}g(y) h(t,y) \; dy\right), 
\end{equation*}
supported in a given open set $\omega\subset\mathbb{T}$, where $g$ is a given nonnegative smooth function satisfying
\[
\int_{\mathbb{T}}g\left(  x\right)  dx=1\quad\text{and}\quad \omega :=\left\{  g>0\right\},
\]
and $h$ is the control input.

Considering the particular cases of the system \eqref{11}, without delay and damping terms, we recall some results from the literature concerning the asymptotic properties. 
The well-known Korteweg--de Vries--Burgers equation
\begin{equation}\label{13}
u_t+u_{xxx}-u_{xx} +uu_x=0  \qtq{in} \RR\times (0,\infty),
\end{equation}
corresponding to the case $j=m=1$, was extensively studied.  For example, in \cite{Amick} Amick \textit{et al.} proved that the $L^2$-norm of the solutions to \eqref{13} tends to zero as $t\rightarrow\infty$ in a polynomial way, namely 
\begin{equation*}
\norm{u(\cdot, t)}_2\le C t^{-\frac 1 2}\qtq{for all} t>0,
\end{equation*}
with a positive constant $C.$ More recently,  Cavalcanti \textit{et al.}  \cite{CCKR} studied the following damped KdV--Burgers equation:
\begin{equation}\label{14}
\begin{cases}
u_t(x,t)+u_{xxx}(x,t)-u_{xx}(x,t) +\lambda_0(x) u(x,t) +u(x,t) u_x(x,t)=0,
&\qtq{in}\RR\times (0,\infty),\\
u(x, 0)=u_0(x),&\qtq{in}\RR.
\end{cases}
\end{equation}
Under appropriate conditions on the damping coefficient $\lambda_0$, the authors established its well-posedness and exponential stability for indefinite damping $\lambda_0(x),$ giving exponential decay estimates on the $L^2-$norm of solutions.
In \cite{gallego}, the results of \cite{CCKR} were extended by generalizing the nonlinear term of \eqref{14} and proving the global well-posedness and exponential stabilization of the generalized KdV-Burger equation under the presence of the localized and indefinite damping. 
In \cite{KomPig2020}, the authors consider the KdV-Burgers equation \eqref{14} in the presence of a delayed feedback $\lambda(x)u(x,t-\tau )$. 
They considered the system with damping and delay feedback, showing the exponential decay estimates under appropriate conditions on the damping coefficients. 

It is important to point out the exponential decay estimates obtained in \cite{MVZ,P} for the KdV equation posed in an interval with localized damping. 
Also, periodic conditions have been considered in \cite{Komornik, KRZ} while more general nonlinearities have been considered in \cite{RZ}.  
Additionally, the robustness concerning the delay of the boundary stability of the nonlinear KdV equation has been studied in \cite{BCV}. 
The authors obtain, under an appropriate condition on the feedback, with and without delay, the local stabilization results for the KdV equation with noncritical length.  
Moreover,  in \cite{Valein}, the authors extended this result for the nonlinear Korteweg-de Vries equation in the presence of an internal delayed term. 

In two recent articles \cite{Chentouf, Chentouf1}, the authors studied the qualitative and numerical analysis of the following nonlinear fourth-order delayed dispersive equation in a bounded domain $I=[0,L]$ with boundary and initial conditions:
\begin{equation}\label{ch}
\begin{cases}u_{t}(x, t)-\sigma u_{x x}(x, t)+\mu u_{x x x x}(x, t)+u(x, t-\tau) u_{x}(x, t)=0, & (x, t) \in I \times(0, \infty), \\ 
u(0, t)=u(L, t)=u_{x}(0, t)=u_{x}(L, t)=0, & t>0, \\ 
u(x, h)=f(x, h), & (x, h) \in I \times[-\tau, 0].
\end{cases}
\end{equation}
The well-posedness, as well as the exponential stability of the zero solution of \eqref{ch}, was established in \cite{Chentouf}. 
The main ingredient of the proof was the exploitation of the Schauder Fixed Point Theorem. 
This improved an earlier result \cite{Chentouf1} in the sense that no interior damping control was required.  
Additionally, numerical simulations were also presented in this article to illustrate the theoretical result. 

Finally, in a recent paper \cite{CaVi} the authors considered the Kawahara equation posed on a bounded interval under the presence of localized damping ($a\left(x\right)u(x,t)$) and delay ($b(x)u(x,t-h)$) terms. They proved its exponential stabilization under suitable assumptions.  First, they showed that the  Kawahara system is exponentially stable under some restriction on the spatial length of the domain. Next, they introduced a more general delayed system and suitable energy functions, they proved, via the Lyapunov approach, the exponential stability for small initial data, under a restriction on the spatial length of the domain.  Then they used a compactness-uniqueness argument to remove these hypotheses, thereby obtaining a semi-global stabilization result.

We end our review with a recent paper \cite{vilmos2022} where abstract linear and nonlinear evolutionary systems with feedback were studied. 
Specifically, the authors considered the system
\begin{equation}\label{general}
\left\{\begin{array}{ll}U^{\prime}(t)=A U(t)+k(t) B U(t-\tau)+F(U(t)),  & \text { in }(0,+\infty). \\ 
U(0)=U_{0}, & \text { for } t \in(0, \tau),\\ 
B U(t-\tau)=f(t), & \end{array}\right.
\end{equation}
where $A$ generates an exponentially stable semigroup in a Hilbert space $H$, $B$ is a continuous linear operator of $H$ into itself, $k(t) \in L^1_{loc}(0,+\infty)$ and $\tau >0$ is a delay parameter, with $F:H\rightarrow H$ is a Lipschitz function. The main purpose of \cite{vilmos2022}  was to give a well-posedness result and an exponential decay estimate for the model \eqref{general} with a damping coefficient $k(t)$ belonging only to $L^1_{loc}$. By using the semigroup approach combined with Gronwall's inequality, under some mild assumptions on the involved functions and parameters, the well-posedness of the problem was established, and an exponential decay estimate was obtained. 

\subsection{Main results}
With this state of the art,  let us now present our main results, which give a necessary next step to understanding the asymptotic behavior for generalized dispersive systems.  

From now on, for the sake of simplicity, the norms in the spaces $L^p(\RR)$ and $L^\infty (\RR)$ will be denoted by  $\Vert \cdot\Vert_p$ and  $\Vert \cdot\Vert_\infty$,   respectively. 
Furthermore, we introduce the  Banach space
\begin{equation*}
{\mathcal{B}}_{s,T}
:= C([0,T];H^s(\RR ))\cap L^2(0,T;H^{s+j}(\RR))
\end{equation*}
with the natural norm
\begin{equation*}
\Vert u\Vert_{\mathcal{B}_{s,T}}=\Vert u\Vert_{ C([0,T];H^s(\RR))}+\Vert \partial^{s+j}_xu\Vert_{L^2(0,T;L^2(\RR))}.
\end{equation*}
Here, we will consider the $H^s(\R)-$norm defined as
\[
\|u\|_{H^s(\R)} = \|(1+\xi^2)^{\frac{s}{2}}\hat{u}(\xi)\|_2,  
\]  
where $\hat{u}$ denotes the Fourier transform of $u$. Moreover, we mention that an alternative equivalent norm for $\|\cdot\|_{H^s(\R)}$ can also be defined using interpolation methods, for details see \cite{lions}. For $s=0$ we omit the subscript $s$, so that $\mathcal{B}_{T}=\mathcal{B}_{0,T}$.  

In the following Theorems and Corollaries, we always assume that
\begin{equation}\label{general-assumptions}
j,m\in\mathbb{N},\quad
m\leq j,\qtq{and}
1\leq p < 2j.
\end{equation}
Our first result ensures the well-posedness of (\ref{11}) in the space $H^s(\R)$:
\begin{Theorem}\label{teo2} 
Assume \eqref{general-assumptions}, and let  $s\in[0,2j+1]$.  
In addition, assume that
\begin{equation*}
\lambda_0, \lambda \in L^{\infty}(\R)\qtq{if}s=0,\qtq{and}
\lambda=0,\ \lambda_0\in H^j(\R)\qtq{if}s>0.
\end{equation*}  
Then, for $s>0$, any $u_0 \in H^s(\R)$ and $T>0$, the IVP (\ref{11}) admits a unique solution $u \in \mathcal{B}_{s,T}$. Moreover, if $s=0$, $u_0 \in C([-\tau,0];L^2(\R))$ and $T>0$, the IVP (\ref{11}) admits a unique solution $u \in \mathcal{B}_{T}$.
Moreover, there exists a nondecreasing continuous function $\beta_s: \R^+ \rightarrow \R^+$, such that
\begin{equation*}
\|u\|_{\B_{s,T}} \leq \beta_s(\|u_0\|_2)\|u_0\|_{H^s(\R)}.
\end{equation*}
\end{Theorem}


\begin{Remark}\label{remark1}
The same result as stated in Theorem \ref{teo2} is obtained when considering the Cauchy problem described in equation \eqref{11} without the presence of a time delay, i.e., when the parameter $\tau$ is set to be zero.
\end{Remark}
Consequently, the subsequent result demonstrates that every mild solution of equation \eqref{11} qualifies as a regular solution when the origin is not taken into account.

\begin{Corollary}\label{corollarywellposedness} Consider $\lambda=0$. Under the assumptions of Theorem \ref{teo2}, for any $u_0 \in L^2(\R)$, the corresponding solution $u$ of \eqref{11} belongs to 
$$
{\mathcal{B}}_{2j+1,[\varepsilon, T]}
:= C([\varepsilon,T];H^{2j+1}(\RR ))\cap L^2(\varepsilon,T;H^{3j+1}(\RR))
$$
for every $0<\varepsilon<T$. 
\end{Corollary}


The second main result is related to the exponential stabilization of  \eqref{11}.  For that,  consider initially the following assumption 
\begin{equation}\label{15}
\lambda_0 (x) \ge \gamma_0 \qtq{for a.e.} x\in \RR
\end{equation}
with some positive constant $\gamma_0.$ If the coefficient of the delay term $\lambda$ satisfies the estimate $\norm{\lambda}_\infty <\gamma_0,$ we can obtain the exponential decay estimates for the  ${\mathcal E}(t)$ associated to the solution of the system \eqref{11}, where  
\begin{equation*} {\mathcal E}(t):= {\mathcal E}(u(t))=\frac 12\int_\RR u^2 (x,t) dx+\frac 12\int_{t-\tau }^t\int_\RR e^{-(t-s)}\vert \lambda (x)\vert u^2 (x,s)\ dx\  ds. 
\end{equation*} 
In that case,  the delay effect is compensated by the undelayed damping term (cf. \cite{NPSicon06}), and the following result holds:

\begin{Theorem}\label{t36} 
Assume \eqref{general-assumptions}. Let $\lambda_0 \in L^\infty (\RR)$ satisfying \eqref{15} and $\lambda\in L^\infty (\RR)$ satisfying
\begin{equation}\label{27}
\frac {e^\tau +1} 2\vert \lambda (x)\vert \le \gamma+ \beta(x)\qtq{for a.e.}x\in\RR
\end{equation}
with $\beta \in L^q(\R)$ for some $1\le q<\infty$, such that
\begin{equation}\label{28}
0\le\gamma< \gamma_0 \qtq{and}
\norm{\beta}_q<\Big(\frac{\gamma_0-\gamma }{c_q}\Big)^{1-\frac 1 {2qj}},  \quad \text{where} \quad c_q:=\Big (1 - \frac 1 {2qj}\Big ) \Big (\frac {2^{2j-1}C^{2j}}{qj} \Big )^{\frac {1} {2qj-1}}.
\end{equation}
Here $C$ is the constant given by the Gagliardo--Nirenberg inequality (see inequality \eqref{e87} below).
Then the solution of \eqref{11}  is exponentially stable. 
More precisely,  the solution $u$ of \eqref{11} satisfies the estimate
\begin{equation*}
{\mathcal E}(t) \le C(u_0,\tau) e^{-\nu_j t}
\end{equation*}
with
\begin{equation}\label{210}
\nu_j = \min \left\{2 \Big (
\gamma_0 -\gamma -c_q
\Vert \beta \Vert_p^{\frac {2qj}{2qj-1}}
\Big ),1\right\}
\end{equation}
and
\begin{equation}\label{211}
C(u_0,\tau)=\frac 12\Vert u(0)\Vert^2_2 +\int_{-\tau}^0 e^s\Vert \lambda\Vert_{\infty} \Vert u(s)\Vert_2^2\, ds.
\end{equation}
\end{Theorem}


After having restricted ourselves to the case where $\lambda_0$ is bounded from below by a positive constant, the next issue is to extend our results to the case where the coefficient of the undelayed feedback $\lambda_0$ is indefinite. We have the following result:
\begin{Theorem}\label{t42}
Assume \eqref{general-assumptions}. Consider $ \lambda\in L^\infty (\RR)$  satisfying  \eqref{27} and $\lambda_0 \in L^\infty (\RR)$ is a indefinite damping, that means that
\begin{equation}\label{41}
\lambda_0 (x) \ge \gamma_0 -\beta_0 (x)  \qtq{for a.e.}x\in \RR ,
\end{equation}
with $\gamma_0>0$ and  $\beta_0, \beta \in L^q(\R)$ for some $1\le q<\infty$, satisfying the inequalities
\begin{equation}\label{42}
\Vert \beta_0\Vert_q<\Big (  \frac {\gamma_0}{c_q} \Big )^{1-\frac 1 {2qj}},
\end{equation}
and
\begin{equation}\label{43}
\Vert  \beta_0 +\beta \Vert_q<\Big (  \frac {\gamma_0-\gamma }{c_q} \Big )^{1-\frac 1 {2qj}}\qtq{for some}0\le\gamma  < \gamma_0.
\end{equation}
Then for every $u_0\in C([-\tau,0]; L^2(\RR))$, the system \eqref{11} has a unique global mild solution.
It is exponentially stable, that is, 
\begin{equation*}
{\mathcal E}(t) \le C(u_0,\tau) e^{-\tilde\nu_j t}
\end{equation*}
with $C(u_0,\tau)$ as in \eqref{211}, and 
\begin{equation}\label{45}
\tilde\nu_j = \min \left\{2 \Big (
\gamma_0 -\gamma  -\frac {2qj-1}{2qj} \Big (\frac 2 q\Big )^{\frac 1 {2qj-1}}
\Vert \beta +\beta_0\Vert_q^{\frac {2qj}{2qj-1}}
\Big ), 1\right \}.
\end{equation}
\end{Theorem}


As a consequence of Theorem \ref{t42}, we can find information on the solution of the system \eqref{11} for any interval $[t,t+T]$, for $T>0$.  The following result will be crucial in getting the exponential stabilization in the space $H^{2j+1}(\R)$.
\begin{Corollary}\label{cor2}
Assume \eqref{general-assumptions}, consider $\lambda_0$ and $\lambda$  satisfying \eqref{27} and \eqref{41}, respectively, and let $T>0$.  
Then for every $u_0 \in C([-\tau,0];L^2(\R))$,
the corresponding solution $u$ of the problem (\ref{11}) satisfies the estimate
\begin{equation*}
\|u\|_{\B_{0,[t,t+T]}}\leq  C_{2T} \left \{ 2C(u_0,\tau)e^{- \nu t}+
  \Vert \lambda\Vert_\infty  \Vert u\Vert_{L^1(t-\tau, t; L^2(\RR))} + \Vert \lambda\Vert_\infty^{1/2}  
 \Vert u\Vert_{L^2(t-\tau, t; L^2(\RR))}  \right \}
\end{equation*}
with $C_T$ is given by \eqref{36} below and $C(u_0,\tau)$ defined by \eqref{211}. 
\end{Corollary}


Considering the general equation \eqref{11} in the $H^s$-norm, with $s \in [0, 2j+1]$ without the presence of the time-delay term, that is, $\lambda=\tau=0$, the previous corollary helps us to prove the following general theorem about the exponential decay:
\begin{Theorem}\label{teoexpHs}
Assume \eqref{general-assumptions}, and consider the system \eqref{11} with $\lambda=\tau=0$, and  $\lambda_0\in L^\infty (\RR)$  satisfying  \eqref{15}.  
Then the system \eqref{11} is exponentially stable. 
More precisely, there exist a time $T_0>0$ such that for every $u_0\in H^s(\RR)$ with $s \in [0,2j+1]$, \eqref{11} has a unique global mild solution satisfying
\begin{equation}\label{eqnew118}
\|u(t)\|_{H^s(\R)} \le \tilde{\gamma}(\|u_0\|_2,T_0) \|u_0\|_{H^{s}(\R)}e^{-\eta(s) t} , \quad \text{for $t\geq T_0$},
\end{equation}
where $\eta(s)$ is a positive number, and $\tilde{\gamma}: \R^{+}\times \R^{+} \rightarrow \R^{+}$ is a continuous function. 
\end{Theorem}


\subsection{Comments and article's structure} 

\begin{itemize}
\item[1.] It is important to point out that to prove Theorem \ref{teo2} we use a method introduced by Tartar \cite{tartar1972interpolation} and adapted by Bona and Scott \cite[Theorem 4.3]{bona1976solutions}. This method was first used to prove the global well-posedness of IVP for the KdV equation on the whole line, and here, we adapt the spaces for our purpose.
\item[2.] Theorems \ref{t36} and \ref{t42} generalize the results  in \cite{KomPig2020} for the general differential operator 
$$
Au:=(-1)^{j+1}\partial_x^{2j+1}u(x,t)+(-1)^m\partial_x^{2m}u(x,t), 
$$
that is, considering appropriate values of $j$, we can recover the results in  \cite{KomPig2020}.  
Additionally to that, considering the real line, the results proved here can be seen as extensions of \cite{Chentouf,Chentouf1,CaVi,CCKR, Valein}, showing the dependence of the decay rate of the order $j$ of the operator $A$.
\item[3.] Note that Theorem \ref{teoexpHs} extends (and recovers) the previous results of \cite{CCKR,gallego} by proving the stabilization in ${\mathcal{B}}_{s, T}$  for $ 0\leq s \leq 2j+1$ with only a damping term, without the presence of the time-delay term. 
\item[4.] The main results of the paper are summarized  in the following table: 
\begin{table}[hbt]
\begin{tabular}{|l|c|c|}
\hline
\multicolumn{1}{|c|}{\textbf{Type of the feedback law}}              & \textbf{Well-posedness} & \textbf{Exponential stabilizition} \\ \hline
Damping (indefined and localized)  term  & ${\mathcal{B}}_{0,T}={\mathcal{B}}_{T}$                       & $\|u(t)\|_{2} \leq C(u_0,\tau) e^{-\eta t} $ \\ 
+ Time Delay  term   &                       &            in $L^2$-norm         \\ \hline
Damping  term                                 & ${\mathcal{B}}_{s,T}$            &         $\|u(t)\|_{H^s(\R)} \leq e^{-\eta t}\|u_0\|_{H^s(\R)}$              \\ 
               (indefined and localized)    & $ 0\leq s \leq 2j+1$           &              $0\leq s \leq 2j+1$          \\ \hline
\end{tabular}
\end{table}
\end{itemize}

This paper consists of five parts, including the introduction. In Section \ref{sectionwellposedness} we analyze the well-posedness of the system \eqref{11} when the initial data belongs to space $C([-\tau,0];H^{s}(\R))$, where $s \in [0,2j+1]$ under some assumptions of $\lambda_0$ and $\lambda$.  
Section \ref{sectionexponentialL2} is devoted to the exponential stabilization of the system \eqref{11} in $L^2(\R)$, that is, to the proof of Theorems \ref{t36} and \ref{t42}.  
In Section \ref{sectionexponentialHsj+1} we prove Theorem \ref{teoexpHs}, establishing the exponential stabilization of the system \eqref{11} without the presence of time delay term ($\lambda=\tau=0$) in the space $H^{s}(\R)$, for $s \in [0,2j+1]$. Finally, in Section \ref{sectionconclusion}, we conclude the work with some further considerations.

\section{Well-posedness theory}\label{sectionwellposedness}
In this section, the strategy is to establish the well-posedness theory in $H^s(\R)$, for $s \in [0,2j+1].$ The strategy is to establish the well-posedness results in $L^2(\R)$ and $H^{2j+1}(\R)$, respectively. After that, we will use interpolation arguments, due to Tartar \cite{tartar1972interpolation} and adapted by Bona and Scott \cite[Theorem 4.3]{bona1976solutions}, to achieve Theorem \ref{teo2}.

\subsection{Linear system}  For simplicity,  consider $m=j\in\mathbb{N}$ in the system \eqref{11}. Let us start studying the following linear model associated with \eqref{11}, namely
\begin{equation}\label{12}
\begin{cases}
u_t(x,t)+(-1)^{j+1}\partial_x^{2j+1}u(x,t)+(-1)^j\partial_x^{2j}u(x,t) +\lambda_0 u(x,t)\\\hspace{4cm}+\lambda u(x,t-\tau )=0,
&\qtq{in}\RR\times (0,\infty),\\
u(x, s)=u_0(x, s),&\qtq{in}\RR\times [-\tau, 0].
\end{cases}
\end{equation}
This section is devoted to proving the well-posedness via \textit{semigroup theory}.
 We first take a look at the properties of the following operator  
\begin{equation}\label{A_op} A_{\lambda_0} u:=-(-1)^{j+1}\partial_x^{2j+1}u(x,t)-(-1)^{j}\partial_x^{2j}u(x,t)-\lambda_0 u.
\end{equation}
 The following well-posedness result can be proved.

\begin{Proposition}\label{p21}
If $\lambda_0\in L^\infty (\RR ),$ then the operator $A_{\lambda_0}$ defined by \eqref{A_op}  on ${\mathcal D}(A_{\lambda_0} ):=H^{2j+1}(\RR)$ generates a strongly continuous semigroup in the Hilbert space $H:=L^2(\RR ).$
\end{Proposition}
\begin{proof}
Note that $A_{\lambda_0}= A +\bar{A}_{\lambda_0}$, where $$Au=-(-1)^{j+1}\partial_x^{2j+1}u(x,t)-(-1)^{j}\partial_x^{2j}u(x,t)$$ and $$\bar{A}_{\lambda_0}= -\lambda_0 u.$$ It suffices to observe that $\bar{A}_{\lambda_0}$ is a bounded perturbation of the operator $A$ and to prove that $A$ generates a strongly continuous semigroup in  $L^2(\RR ).$ Indeed, 
according to the Lumer-Phillips theorem it is sufficient to check that $A $ is dissipative and that $I-A $ is onto. The dissipativity follows by a direct computation: if $u_{0} \in H^{2j+1}(\mathbb{R})$ is real-valued, then $u$ is also real-valued and
\begin{equation*}
\begin{aligned}
(A  u, u)_{H} &=\int_{-\infty}^{\infty}\left(-(-1)^{j+1}\partial_x^{2j+1}u-(-1)^{j}\partial_x^{2j}u\right) u d x =-\int_{-\infty}^{\infty}( \partial^j_x u)^2d x  \leqslant 0,
\end{aligned}
\end{equation*}
since we have 
$$\int_{-\infty}^{\infty}(-1)^{j+1}(\partial_x^{2j+1}u)u dx=0.$$
Thanks to the fact that  $\Re A  v=A (\Re v)$ for all $v \in H^{2j+1}(\mathbb{R})$, it follows that
$$
\Re(A  u, u)_{H}=-\int_{-\infty}^{\infty}( \partial^j_x u)^2d x \leqslant 0
$$
for all $u_{0} \in H^{2j+1}(\mathbb{R})$.

In order to prove that $I-A$ is onto, it remains to show that for every $f \in L^{2}(\mathbb{R})$ there exists $u \in H^{2j+1}(\mathbb{R})$ satisfying the equality
$$
(-1)^{j+1}\partial_x^{2j+1}u(x,t)+(-1)^{j}\partial_x^{2j}u(x,t)+ u=f.
$$
Taking the Fourier transform in the previous equation, it is equivalent to
$$
\hat{u}(\xi)=\frac{\hat{f}(\xi)}{1+((-1)^{j}(i\xi)^{2j})+(-1)^{j+1}(i\xi)^{2j+1}}.
$$
This is possible because the denominator
$$
h(\xi):=1+((-1)^{j}(i\xi)^{2j})+(-1)^{j+1}(i\xi)^{2j+1}
$$
never vanishes.  Since, moreover, $h(\xi)$ is a continuous function satisfying $|h(\xi)| \rightarrow \infty$ as $|\xi| \rightarrow \infty, 1 / h$ is bounded, and therefore the last equation has a unique solution $\hat{u} \in L^{2}(\mathbb{R})$.
Finally, since the function
$$\frac{(1+\xi^2)^{2j+1}}{|h(\xi)|^2}=\frac{(1+\xi^2)^{2j+1}}{\left|1+((-1)^{j}(i\xi)^{2j})+(-1)^{j+1}(i\xi)^{2j+1}\right|}$$
tends to 1 as $|\xi| \rightarrow \infty$ and hence it is bounded by some constant $M$ on $\mathbb{R}$, we conclude that
$$\|u\|_{H^{2j+1}(\R)}^2=\|(1+ \xi^2)^{\frac{2j+1}{2}} \hat{u}\|^2_2 \leqslant M\|\hat{f}\|_2^2.$$
Since $\hat{f} \in L^{2}(\mathbb{R})$, this implies the regularity property $u \in H^{2j+1}(\mathbb{R})$.
\end{proof}

Let us now use an iterative procedure (see e.g. \cite{NP15}) and semigroup theory (see e.g. \cite{Pazy}),  to prove that \eqref{12} is well-posed.

\begin{Theorem}\label{t22}
If $\lambda_0,\lambda\in L^\infty (\RR )$ and $u_0\in C([-\tau ,0]; H)$,  then there exists a unique solution $u\in C([-\tau, +\infty); H)$ of the problem \eqref{12}.
\end{Theorem}

\begin{proof}
Consider the interval $[0,\tau].$ Note that \eqref{12}  can be seen as an inhomogeneous Cauchy problem of the form
\begin{equation*}
\begin{cases}
u_t(t)-A_{\lambda_0}u(t)=g_0(t),&\qtq{in}(0,\tau ),\\
u(0)=u_0,
\end{cases}
\end{equation*}
with $g_0(t)=-\lambda u_0(t-\tau),$ for $t\in [0,\tau],$ which have unique solution $u(\cdot) \in C([0,\tau), H).$ Now, considering the interval $t\in [\tau, 2\tau ]$,  problem \eqref{12} can be rewritten as
\begin{equation*}
\begin{cases}
u_t(t)-A_{\lambda_0}u(t)=g_1(t),&\qtq{in}(\tau,2\tau ),\\
u(\tau )=u({\tau}_{-}),
\end{cases}
\end{equation*}
where $g_1(t)=-\lambda u(t-\tau).$ Thanks to the first step of the proof the function $u(t)$, for $t\in [0,\tau]$, its known,  thus $g_1(t)$ can be considered as a known function for $t\in [\tau , 2\tau].$ Therefore, this analysis yields the existence of a solution $u(\cdot )\in C ([0, 2\tau ], H).$ By a bootstrap argument we get a solution $u\in C([0, \infty ), H).$
\end{proof}

\subsection{Non-homogeneous system}

We are interested in extending the previous results to the nonlinear system \eqref{11}. In this way, we first consider the corresponding linear inhomogeneous initial value problem
\begin{equation}\label{31}
\begin{cases}
u_t(x,t)+(-1)^{j+1}\partial_x^{2j+1}u(x,t)+(-1)^j\partial_x^{2j}u(x,t) +\lambda_0 u(x,t)\\\hspace{4cm}+\lambda u(x,t-\tau )=f(x,t), &\qtq{in}\RR\times (0,T),\\
u(x,s)=u_0(x,s), &\qtq{in}\RR\times [-\tau, 0],
\end{cases}
\end{equation}
for some $T>0.$ Consider the operator $A_{\lambda_0}$ defined by Proposition \ref{p21}. So, we may rewrite \eqref{31} in the following way
\begin{equation}\label{32}
\begin{cases}
u_t(x,t)+\lambda u(x,t-\tau )=A_{\lambda_0} u(x,t) +f(x,t),&\qtq{in}\RR\times (0,T),\\
u(x,s)=u_0(x,s),&\qtq{in}\RR\times [-\tau, 0].
\end{cases}
\end{equation}

Since $A_{\lambda_0}$ generates a strongly continuous semigroup of contractions in $L^2(\RR )$ by Proposition \ref{p21}, for any given value $u_0\in C([-\tau, 0], H)$ and $f\in L^1(0,T;L^2(\RR)),$ problem \eqref{32} has a unique mild solution $u\in C([-\tau, T];L^2(\RR)),$ satisfying  Duhamel's formula:
\begin{equation}\label{33}
u(t)=S(t)u_0(0)-\int_0^t S(t-s)\lambda u(s-\tau )\, ds +\int_0^t S(t-s) f(s) \, ds,\quad t\in [0,T].
\end{equation}
Therefore, we can prove that the mild solution of \eqref{32} depends continuously on the initial data.

\begin{Proposition}
If $u_0\in C([-\tau, 0], L^2(\RR))$ and $f\in L^1(0,T;L^2(\RR)),$ then the solution of  \eqref{32}
satisfies the following estimate:
\begin{equation}\label{34}
\Vert u\Vert_
{C([0,T];L^2(\RR))} \le
e^{\Vert\lambda\Vert_\infty T}
\left (\left\|u_0(0)\right\|_2+\Vert f\Vert_{L^1(0,T;L^2(\RR))}+
\Vert \lambda\Vert_\infty\int_{-\tau }^0\Vert u(s)\Vert_2 \, ds \right )
\end{equation}
and
\begin{equation}\label{35_1}
\Vert u\Vert_
{C([-\tau,T];L^2(\RR))} \le
C
\left ( \Vert u_0\Vert_{C([-\tau,0];L^2(\RR))}+\Vert f\Vert_{L^1(0,T;L^2(\RR))}\right ),
\end{equation}
where $\lambda$ is a function in $L^{\infty}(\R)$ (see \eqref{27}) and $C$ is a positive constant given by $C=C(\|\lambda\|_{\infty}, T, \tau)$.
\end{Proposition}

\begin{proof}
The Duhamel's formula \eqref{33} give us that
\begin{align*}
\Vert u(t)\Vert_2
\le & \left\|u_0(0)\right\|_2
+\Vert \lambda\Vert_\infty\int_0^t\Vert u(s-\tau )\Vert_2\, ds+\Vert f\Vert_{L^1(0,T;L^2(\RR))}\\
\le & \left\|u_0(0)\right\|_2+\Vert f\Vert_{L^1(0,T;L^2(\RR))}+\Vert \lambda\Vert_\infty\int_{-\tau}^0\Vert u(s)\Vert_2\, ds+\Vert \lambda\Vert_\infty\int_0^t\Vert u(s)\Vert_2\, ds \\ 
=  & \left\|u_0(0)\right\|_2+\Vert f\Vert_{L^1(0,T;L^2(\RR))}+\Vert \lambda\Vert_\infty\int_{-\tau}^t\Vert u(s)\Vert_2\, ds . 
\end{align*}
Thus, the result follows as a direct application of Gronwall's lemma.
\end{proof}

As a consequence of the previous inequality, we have the following proposition.
\begin{Proposition}\label{p32}
Let $u_0\in C([-\tau, 0],L^2(\RR))$ and $f\in L^1(0,T;L^2(\RR)),$ then the solution of  \eqref{32} belongs to ${\mathcal B}_T$ and the following estimate holds true:
\begin{equation}\label{35}
\begin{split}
\Vert u\Vert_{{\mathcal B}_T}\le C_T \left \{  \left\|u_0(0)\right\|_2+\Vert f\Vert_{L^1(0,T;L^2(\RR))}+
  \Vert \lambda\Vert_\infty  \Vert u\Vert_{L^1(-\tau, 0; L^2(\RR))}\right.\\ \left. + \Vert \lambda\Vert_\infty^{1/2}  
 \Vert u\Vert_{L^2(-\tau, 0; L^2(\RR))}  \right \}
 \end{split}
\end{equation}
where
\begin{equation}\label{36}
C_T=\sqrt{\frac 3 2} \left (1+e^{2\Vert \lambda\Vert_\infty T}\right )^{1/2}
e^{(\Vert \lambda\Vert_\infty +\Vert \lambda_0\Vert_\infty) T}.
\end{equation}
Moreover, we have 
\begin{align}\label{37}
\begin{split}
\frac 12 \Vert u(t)\Vert_2^2+\int_0^t\Vert\partial^j_xu\Vert^2_2 ds&+\int_0^t\int_\RR \lambda_0 u^2(x,s)dx ds
+\int_0^t\int_\RR\lambda u(x, s-\tau )u(x, s)dxds \\
&=\frac 12 \left\|u_0(0)\right\|_2+\int_0^t\int_\RR f(x,s) u(x,s) dxds,
\end{split}
\end{align}
for all $t\in [0,T]$,
\end{Proposition}
\begin{proof}
Multiplying the equation \eqref{31} by $u$ and integrating by parts, taking into account that  
$$\int_{\mathbb{R}}(-1)^{j+1}(\partial_x^{2j+1}u)u dx=0 \quad \text{and}\quad \int_{\mathbb{R}}(-1)^j\partial_x^{2j}u(x,t)u dx=\int_{\mathbb{R}}(\partial^j_xu)^2dx,$$
the relation \eqref{37} holds.   

Let us now prove \eqref{35}.  Thanks to \eqref{34} we infer that
\begin{align*}
\begin{split}
\Vert u(t)\Vert_2^2
&+2\int_0^t \Vert \partial^j_xu\Vert_2^2 ds
\le \left\|u_0(0)\right\|_2^2\\
&+2\Vert f\Vert_{L^1(0,T;L^2(\RR))} e^{\Vert \lambda\Vert_\infty T}\left (  \left\|u_0(0)\right\|_2+\Vert f\Vert_{L^1(0,T;L^2(\RR))}+
\Vert \lambda\Vert_\infty\int_{-\tau }^0\Vert u(s)\Vert_2 \, ds \right ) \\
&+2\Vert\lambda_0\Vert_\infty \int_0^t\Vert u(s)\Vert_2^2ds +\Vert\lambda\Vert_\infty\int_0^t\Vert u(s-\tau )\Vert_2^2ds +\Vert\lambda\Vert_\infty\int_0^t\Vert u(s)\Vert_2^2ds.
\end{split}
\end{align*}
The previous inequality, together with the following one
\begin{equation*}
\int_0^t\Vert u(s-\tau )\Vert_2^2 \, ds\le \int_{-\tau}^0\Vert u(s)\Vert_2^2 \, ds +\int_0^t\Vert u(s)\Vert_2^2\, ds,
\end{equation*}
ensures that
\begin{align}\label{38}
\begin{split}
\Vert u(t)\Vert_2^2 &+2\int_0^t \Vert \partial^j_xu\Vert_2^2ds \le\left\|u_0(0)\right\|_2^2+\Vert f\Vert_{L^1(0,T;L^2(\RR))}^2\\
&+e^{2\Vert \lambda\Vert_\infty T}\left ( \left\|u_0(0)\right\|_2+\Vert f\Vert_{L^1(0,T;L^2(\RR))}+\Vert \lambda\Vert_\infty \Vert u\Vert_{L^1(-\tau, 0; L^2(\RR))}  \right )^2 \\
&+\Vert \lambda\Vert_\infty
\Vert u\Vert_{L^2(-\tau, 0; L^2(\RR))}^2
+ 2(\Vert \lambda\Vert_\infty +\Vert \lambda_0\Vert_\infty)
\int_0^t\Vert u\Vert^2_2ds.
\end{split}
\end{align}
From \eqref{38} we have
\begin{equation}\label{help}
\begin{split}
\Vert u(t)\Vert_2^2 &+2\int_0^t \Vert \partial^j_xu\Vert_2^2ds\le 2( \Vert \lambda\Vert_\infty +\Vert \lambda_0\Vert_\infty)
\int_0^t\Vert u\Vert^2_2ds\\
&+\Big (1+e^{2\Vert \lambda\Vert_\infty T}\Big )
\left\{  \left\|u_0(0)\right\|_2+\Vert f\Vert_{L^1(0,T;L^2(\RR))} \right.\\
&\left.  +\Vert \lambda\Vert_\infty  \Vert u\Vert_{L^1(-\tau, 0; L^2(\RR))} + \Vert \lambda\Vert_\infty^{1/2}  
 \Vert u\Vert_{L^2(-\tau, 0; L^2(\RR))}  \right\}^2.
\end{split}
\end{equation}
An application of Gronwall's lemma gives the following:
\begin{align*}
\begin{split}
\Vert u(t)\Vert_2^2 &+2\int_0^t \Vert \partial^j_xu\Vert_2^2\,  ds
\le \left (1+e^{2\Vert \lambda\Vert_\infty T}\right )
 e^{2(\Vert \lambda\Vert_\infty +\Vert \lambda_0\Vert_\infty) T}\times\\
&\times
\Big \{\left\|u_0(0)\right\|_2+\Vert f\Vert_{L^1(0,T;L^2(\RR))}+
  \Vert \lambda\Vert_\infty  \Vert u\Vert_{L^1(-\tau, 0; L^2(\RR))} + \Vert \lambda\Vert_\infty^{1/2}  
 \Vert u\Vert_{L^2(-\tau, 0; L^2(\RR))}   \Big \}^2,
\end{split}
\end{align*}
thus
\begin{align*}
\begin{split}
\Vert u\Vert_{{\mathcal B}_T}^2\le&
\frac 3 2 \left (1+e^{2\Vert \lambda\Vert_\infty T}\right )
e^{2(\Vert \lambda\Vert_\infty +\Vert \lambda_0\Vert_\infty) T}
\times\\
&\times
\Big \{ \left\|u_0(0)\right\|_2+\Vert f\Vert_{L^1(0,T;L^2(\RR))}+
  \Vert \lambda\Vert_\infty  \Vert u\Vert_{L^1(-\tau, 0; L^2(\RR))} + \Vert \lambda\Vert_\infty^{1/2}  
 \Vert u\Vert_{L^2(-\tau, 0; L^2(\RR))} \Big \}^2,
\end{split}
\end{align*}
and so
\begin{align*}
\begin{split}
\Vert u\Vert_{{\mathcal B}_T}\le&
\sqrt{\frac 3 2} \left (1+e^{2\Vert \lambda\Vert_\infty T}\right )^{1/2}
e^{(\Vert \lambda\Vert_\infty +\Vert \lambda_0\Vert_\infty) T}
\times\\
&\times
\Big \{ \left\|u_0(0)\right\|_2+\Vert f\Vert_{L^1(0,T;L^2(\RR))}+
  \Vert \lambda\Vert_\infty  \Vert u\Vert_{L^1(-\tau, 0; L^2(\RR))} + \Vert \lambda\Vert_\infty^{1/2}  
 \Vert u\Vert_{L^2(-\tau, 0; L^2(\RR))}   \Big \}
\end{split}
\end{align*}
showing \eqref{35} with $C_T$ defined by \eqref{36}. 
\end{proof}

\subsection{Nonlinear estimates} In this subsection, we present some nonlinear estimates that will be used to prove Theorem \ref{teo2}.

\begin{Lemma}\label{lm1}
Let $1\leq p < 2j$ with $j\geq 1$. Then, there exists a positive constant $C$, such that, for any $T>0$ and $u,v \in \mathcal{B}_{T}$, we have
\begin{equation*}
\|u^pv_x\|_{L^1(0,T;L^2(\R))}\leq 2^{\frac{p}{2}}CT^{\frac{4j-p-2}{4j}} 
\|u\|_{\mathcal{B}_T}^p\|v\|_{\mathcal{B}_T}.
\end{equation*}
\end{Lemma}

\begin{proof}
Recall that $H^{j}(\R)\hookrightarrow  H^1(\R) \hookrightarrow L^{\infty}(\R)$ for $j\geq 1$.
On the other hand, note that the following inequality holds:
\begin{equation}\label{215}
\Vert v\Vert_\infty^2\le 2\Vert v\Vert_2\Vert  v_x\Vert_{2},
\end{equation}
for all $v\in H^j(\RR)$.  Indeed, if $v \in C_{c}^{\infty}(\mathbb{R})$ and $y \in \mathbb{R}$, then we have the following inequality:
\begin{equation}\label{eqnew1}
\left|v(y)^{2}\right|=\left|\int_{-\infty}^{y} 2 v v_{x} d x\right| \leqslant 2 \int_{-\infty}^{\infty}|v| \cdot\left|v_{x}\right| d x \leqslant 2\|v\|_{2}\left\|v_{x}\right\|_{2},
\end{equation}
proving our estimate for smooth functions. Note that the general case, that is, for $v\in H^j(\RR)$, follows by density.  Thus, we have that
\begin{align*}
\|u^pv_x\|_{L^1(0,T;L^2(\R))}&\leq  C\int _{0}^{T}\|u(t)\|_{\infty}^p\|v_x(t)\|_2dt \\
&\leq  2^{\frac{p}{2}}C\int _{0}^{T}\|u(t)\|_{2}^{\frac{p}{2}}\|u_x(t)\|_{2}^{\frac{p}{2}}\|v_x(t)\|_2dt \\
&\leq  2^{\frac{p}{2}}C\|u\|_{C([0,T];L^2)}^{\frac{p}{2}}\int _{0}^{T}\|u_x(t)\|_{2}^{\frac{p}{2}}\|v_x(t)\|_2dt.
\end{align*}
We recall the following case of the Gagliardo--Nirenberg inequality:
\begin{equation}\label{e87}
\|\partial^m_x\varphi\|_2 \leq C \|\partial^j_x\varphi\|^{\frac{m}{j}}_2\|\varphi\|_2^{1-\frac{m}{j}}, \quad m\leq j.
\end{equation}
Hence, taking into account \eqref{215} and considering $m=1$ in \eqref{e87}, we get that
\begin{equation}\label{215_1}
\Vert \varphi\Vert_\infty^2\le 2\Vert \varphi\Vert_2\Vert  \varphi_x\Vert_{2}\le 2C \Vert \partial_x^j \varphi\Vert_{2}^{\frac{1}{j}}\Vert  \varphi\Vert_2^{2-\frac{1}{j}}.
\end{equation}
Moreover, by using \eqref{e87} and H\"older inequality,  we obtain
\begin{align*}
\|u^pv_x\|_{L^1(0,T;L^2(\R))}&\leq   2^{\frac{p}{2}}C\|u\|_{C([0,T];L^2)}^{\frac{p}{2}}\int _{0}^{T}\|\partial_x^j u(t)\|_{2}^{\frac{p}{2j}}\|u(t)\|_{2}^{\frac{p}{2}\left( 1- \frac{1}{j}\right)}\|v_x(t)\|_2dt \\
&\leq   2^{\frac{p}{2}}C\|u\|_{C([0,T];L^2)}^{\frac{p}{2}\left( 2- \frac{1}{j}\right)} \left( \int _{0}^{T}\|\partial_x^j u(t)\|_{2}^{\frac{p}{2j} . \frac{4j}{p}} \right)^{\frac{p}{4j}} \left( \int _{0}^{T} \|v_x(t)\|_2^2dt \right)^{\frac12}\left( \int _{0}^{T} dt \right)^{\frac{2j-p}{4j}} \\
&\leq   2^{\frac{p}{2}}C T^{\frac{2j-p}{4j}} \|u\|_{C([0,T];L^2)}^{\frac{p}{2}\left( 2- \frac{1}{j}\right)} \|\partial_x^ju\|_{L^2([0,T];L^2)}^{\frac{p}{2j}}\left( \int _{0}^{T} \|\partial_x^jv(t)\|_2^{\frac{2}{j}}\|v(t)\|_2^{2\left(1-\frac{1}{j}\right)}dt \right)^{\frac12}  \\
&\leq   2^{\frac{p}{2}}C T^{\frac{2j-p}{4j}} \|u\|_{C([0,T];L^2)}^{\frac{p}{2}\left( 2- \frac{1}{j}\right)} \|\partial_x^ju\|_{L^2([0,T];L^2)}^{\frac{p}{2j}}\| v\|_{\mathcal{B}_T}^{1-\frac{1}{j}}\left( \int _{0}^{T} \|\partial_x^jv(t)\|_2^{\frac{2}{j}}dt \right)^{\frac12}  \\
&\leq   2^{\frac{p}{2}}C T^{\frac{2j-p}{4j}} \|u\|_{C([0,T];L^2)}^{\frac{p}{2}\left( 2- \frac{1}{j}\right)} \|\partial_x^ju\|_{L^2([0,T];L^2)}^{\frac{p}{2j}} \| v\|_{\mathcal{B}_T}^{1-\frac{1}{j}}\left( \|\partial_x^jv(t)\|_{L^2([0,T];L^2)}^{\frac{1}{j}} T^{\frac12\left(1-\frac{1}{j}\right)}\right) \\
&\leq   2^{\frac{p}{2}}C T^{\frac{2j-p}{4j}}T^{\frac12\left(1-\frac{1}{j}\right)} \|u\|_{\mathcal{B}_T}^{p}   \| v\|_{\mathcal{B}_T},
\end{align*}
proving the result.
\end{proof}

The next lemma gives the estimates for the nonlinear terms. 
\begin{Lemma}\label{lm2}
For any $T>0$, $1\leq p <2j$,  $\lambda_0 \in L^{\infty}(\R)$ and $u, v, w \in \mathcal{B}_{T}$, we have
\begin{enumerate}
\item[(i)] $\|\lambda_0 u\|_{L^1(0,T;L^2(\R))} \leq T^{\frac{1}{2}}\|\lambda_0\|_{\infty}\|u\|_{\mathcal{B}_{T}}$;
\item[(ii)] $\|uw_x\|_{L^1(0,T;L^2(\R))} \leq CT^{\frac{4j-3}{4j}}\|u\|_{\mathcal{B}_T}\|w\|_{\mathcal{B}_{T}}$;
\item[(iii)] $\|u|v|^{p-1}w_x\|_{L^1(0,T;L^2(\R))} \leq 2^{\frac{p}{2}}CT^{\frac{2j-p}{4j}}T^{\frac{1}{2}\left(1-\frac{1}{j}\right)}\|u\|_{\mathcal{B}_{T}}\|w\|_{\mathcal{B}_{T}}\|v\|_{\mathcal{B}_{T}}^{p-1}$;
\item[(iv)]The map $$M: \mathcal{B}_{T} \rightarrow L^1(0,T;L^2(\R)),$$ defined by $Mu:=u^p u_x$,  is locally Lipschitz continuous and satisfies the inequality
\begin{equation*}
\begin{split}
\|Mu-Mv\| |_{L^1(0,T;L^2(\R))} \leq& 2^{\frac{p}{2}}T^{\frac{2j-p}{4j}}T^{\frac{1}{2}\left(1-\frac{1}{j}\right)}C\left(\|u\|_{\mathcal{B}_T}^p+\|u\|_{\mathcal{B}_T}\|v\|_{\mathcal{B}_T}^{p-1} +\|v\|_{\mathcal{B}_T}^{p}\right)\|u-v\|_{\mathcal{B}_T}\\ &+2^{\frac{1}{2}}T^{\frac{4j-3}{4j}}C\|u\|_{\mathcal{B}_T}\|u-v\|_{\mathcal{B}_T} ,
\end{split}
\end{equation*}
with a positive constant $C$.
\end{enumerate}
\end{Lemma}
\begin{proof}
Note that (i) follows using H\"{o}lder inequality.  
For (ii), observe that thanks to \eqref{215} and \eqref{e87},  we have
\begin{align*}
\|uw_x\|_{L^1(0,T;L^2(\R))} 
&\leq \int_0^T \|u(t)\|_{\infty}\|w_x(t)\|_2dt \leq C\int_0^T\|u(t)\|_{2}^{\frac{1}{2}\left(2-\frac{1}{j}\right)}\|\partial_x^ju(t)\|_{2}^{\frac{1}{2j}}\|w_x(t)\|_2dt \\
& \leq C\|u\|_{\mathcal{B}_T}^{\frac{1}{2}\left(2-\frac{1}{j}\right)}\int_0^T\|\partial_x^ju(t)\|_{2}^{\frac{1}{2j}}\|w(t)\|_2^{1-\frac{1}{j}}\|\partial_x^j w(t)\|_2^{\frac{1}{j}}dt \\
& \leq C\|u\|_{\mathcal{B}_T}^{\frac{1}{2}\left(2-\frac{1}{j}\right)}\|w\|_{\mathcal{B}_T}^{1-\frac{1}{j}}\int_0^T\|\partial_x^ju(t)\|_{2}^{\frac{1}{2j}}\partial_x^j \|w(t)\|_2^{\frac{1}{j}}dt \\
&\leq C\|u\|_{\mathcal{B}_T}^{\frac{1}{2}\left(2-\frac{1}{j}\right)}\|w\|_{\mathcal{B}_T}^{1-\frac{1}{j}}\left(\int_0^T\|\partial_x^ju(t)\|_{2}^{2}dt\right)^{\frac{1}{4j}}\left(\int_0^T\|\partial_x^j w(t)\|_{2}^{2}dt\right)^{\frac{1}{2j}}T^{\frac{4j-3}{4j}} \\
&\leq CT^{\frac{4j-3}{4j}}\|u\|_{\mathcal{B}_T}^{\frac{1}{2}\left(2-\frac{1}{j}\right)}\|w\|_{\mathcal{B}_T}^{1-\frac{1}{j}}\|u\|_{\mathcal{B}_T}^{\frac{1}{2j}}\|w\|_{\mathcal{B}_T}^{\frac{1}{j}}.
\end{align*}
Let us now prove (iii).  Note that \eqref{e87} implies that
\begin{align*}
&\|u|v|^{p-1}w_x\|_{L^1(0,T;L^2(\R))} \\
&\leq 2^{\frac{p}{2}}\int_0^T \|u(t)\|_{2}^{\frac{1}{2}}\|u_x(t)\|_{2}^{\frac{1}{2}}\|v(t)\|_{2}^{\frac{p-1}{2}}\|v_x(t)\|_{2}^{\frac{p-1}{2}}\|w_x(t)\|_2dt \\
&\leq 2^{\frac{p}{2}}C\|u\|_{\mathcal{B}_T}^{\frac{1}{2}}\|v\|_{\mathcal{B}_T}^{\frac{p-1}{2}}\int_0^T \|\partial_x^j u(t)\|_{2}^{\frac{1}{2j}}\|u(t)\|_{2}^{\frac{1}{2}\left( 1- \frac{1}{j}\right)}\|\partial_x^j v(t)\|_{2}^{\frac{p-1}{2j}}\|v(t)\|_{2}^{\frac{p-1}{2}\left( 1- \frac{1}{j}\right)}\|w_x(t)\|_{2}dt \\
&\leq 2^{\frac{p}{2}}C\|u\|_{\mathcal{B}_T}^{\frac{1}{2}}\|v\|_{\mathcal{B}_T}^{\frac{p-1}{2}} \|u\|_{\mathcal{B}_T}^{\frac{1}{2}\left( 1- \frac{1}{j}\right)} \|v\|_{\mathcal{B}_T}^{\frac{p-1}{2}\left( 1- \frac{1}{j}\right)} \int_0^T \|\partial_x^j u(t)\|_{2}^{\frac{1}{2j}}\|\partial_x^j v(t)\|_{2}^{\frac{p-1}{2j}}\|w_x(t)\|_{2}dt \\
&= 2^{\frac{p}{2}}C \|u\|_{\mathcal{B}_T}^{\frac{1}{2}\left( 2- \frac{1}{j}\right)} \|v\|_{\mathcal{B}_T}^{\frac{p-1}{2}\left( 2- \frac{1}{j}\right)} \int_0^T \|\partial_x^j u(t)\|_{2}^{\frac{1}{2j}}\|\partial_x^j v(t)\|_{2}^{\frac{p-1}{2j}}\|w_x(t)\|_{2}dt. 
\end{align*}
Thus,  H\"older's inequality ensures that
\begin{multline*}
\|u|v|^{p-1}w_x\|_{L^1(0,T;L^2(\R))} \\
\leq 2^{\frac{p}{2}} C \|u\|_{\mathcal{B}_T}^{\frac{1}{2}\left( 2- \frac{1}{j}\right)} \|v\|_{\mathcal{B}_T}^{\frac{p-1}{2}\left( 2- \frac{1}{j}\right)} \left(\int_0^T \|\partial_x^j u\|_{2}^{2}dt\right)^{\frac{1}{4j}}\left(\int_0^T\|\partial_x^j v\|_{2}^{2}dt\right)^{\frac{p-1}{4j}}\left(\int_0^T\|w_x\|_{2}^{2}dt\right)^{\frac{1}{2}}T^{\frac{2j-p}{4j}}.
\end{multline*}
The previous inequality and \eqref{e87} give us
\begin{align*}
&\|u|v|^{p-1}w_x\|_{L^1(0,T;L^2(\R))} \\
&\qquad\leq 2^{\frac{p}{2}}CT^{\frac{2j-p}{4j}} \|u\|_{\mathcal{B}_T}^{\frac{1}{2}\left( 2- \frac{1}{j}\right)} \|v\|_{\mathcal{B}_T}^{\frac{p-1}{2}\left( 2- \frac{1}{j}\right)}\|u\|_{\mathcal{B}_T}^{\frac{1}{2j}} \|v\|_{\mathcal{B}_T}^{\frac{p-1}{2j}}  \left(\int_0^T \| w\|_{2}^{2\left(1-\frac{1}{j}\right)}\| \partial_x^jw\|_{2}^{\frac{2}{j}}dt\right)^{\frac{1}{2}}\\
&\qquad\leq 2^{\frac{p}{2}} C T^{\frac{2j-p}{4j}} \|u\|_{\mathcal{B}_T} \|v\|_{\mathcal{B}_T}^{p-1} \|w\|_{\mathcal{B}_T}^{1-\frac{1}{j}}\left(\int_0^T \| \partial_x^jw\|_{2}^{2}dt\right)^{\frac{1}{2j}}T^{\frac{1}{2}\left(1-\frac{1}{j}\right)} \\
&\qquad\leq 2^{\frac{p}{2}} C T^{\frac{2j-p}{4j}} T^{\frac{1}{2}\left(1-\frac{1}{j}\right)}\|u\|_{\mathcal{B}_T} \|v\|_{\mathcal{B}_T}^{p-1} \|w\|_{\mathcal{B}_T}^{1-\frac{1}{j}}\|w\|_{\mathcal{B}_T}^{\frac{1}{j}},
\end{align*}
which allows us to get (iii).  Finally, using the mean value theorem,  (ii),  (iii), and Lemma \ref{lm1}, we have
\begin{align*}
&\|Mu -Mv \|_{L^1(0,T;L^2)}\\
&\qquad\leq C\|(1+|u|^{p-1}+|v|^{p-1})|u-v|u_x\|_{L^1(0,T;L^2)}+\|v^p(u-v)_x\|_{L^1(0,T;L^2)} \\
&\qquad\leq C\left\lbrace T^{\frac{4j-3}{4j}}\|u-v\|_{\mathcal{B}_T}\|u\|_{\mathcal{B}_T} +2^{\frac{p}{2}}T^{\frac{2j-p}{4j}}T^{\frac{1}{2}\left(1-\frac{1}{j}\right)}\|u-v\|_{\mathcal{B}_T}\|u\|_{\mathcal{B}_T}^p\right. \\
&\qquad\qquad\left.+2^{\frac{p}{2}}T^{\frac{2j-p}{4j}}T^{\frac{1}{2}\left(1-\frac{1}{j}\right)}\|u-v\|_{\mathcal{B}_T}\|u\|_{\mathcal{B}_T}\|v\|_{\mathcal{B}_T}^{p-1}
+2^{\frac{p}{2}}T^{\frac{2j-p}{4j}}T^{\frac{1}{2}\left(1-\frac{1}{j}\right)}\|u-v\|_{\mathcal{B}_T}\|v\|_{\mathcal{B}_T}^{p} \right\rbrace,
\end{align*}
and (iv) holds.
\end{proof}

\subsection{Nonlinear system} We are in a position to consider the nonlinear model \eqref{11}, with $u_0\in C([-\tau, 0]; L^2(\RR)).$ Before presenting it, let us introduce the following definitions: 
\begin{Definition}
A mild solution of \eqref{11} is a function $u\in {\mathcal B}_T,$ $T>0,$ which satisfies
\begin{align*}
u(t)&=S(t) u_0(0) -\int_0^t S(t-s)\lambda u(s-\tau )ds -\int_0^t S(t-s) u^p(s)\partial_x u(s) ds,\quad t\in [0,T].
\end{align*}
A global mild solution of \eqref{11} is a function $u:[0,\infty)\rightarrow H^1(\RR)$
whose restriction to every bounded interval $[0,T]$ is a mild solution of \eqref{11}.
\end{Definition}

\subsubsection{\textbf{Well-posedness theory in $L^2({\R})$\label{s3}}}
With these definitions and the previous estimates in hand, the following result gives the local well-posedness for the higher-order dispersive equation and an a priori estimate for the solutions of \eqref{11}.
\begin{Proposition}\label{r1}
Let  $1\leq p<2j$ with $j\geq 1$, and $\lambda_0, \lambda \in L^{\infty}(\R)$.  For $u_0 \in L^2(\R)$,  there exist $T>0$ and a unique mild solution $u \in \mathcal{B}_{T}$ of \eqref{11}, such that
\begin{equation*}
\Vert u\Vert_{{\mathcal B}_T}\le C_T \left \{ \Vert u(0)\Vert_2 +
  \Vert \lambda\Vert_\infty  \Vert u_0\Vert_{L^1(-\tau, 0; L^2(\RR))} + \Vert \lambda\Vert_\infty^{1/2}  
 \Vert u_0\Vert_{L^2(-\tau, 0; L^2(\RR))}  \right \}.
\end{equation*}
Here, $C_T$ is given by \eqref{36}.
\end{Proposition}

\begin{proof}
Let $T>0$ be determined later and consider the  operator $\Gamma:\mathcal{B}_{T}\longrightarrow \mathcal{B}_{T}$ given by $\Gamma (u) = v$. For each $u \in \mathcal{B}_{T}$ consider the problem
\begin{equation}\label{e9}
\begin{cases}
v_t = A_{\lambda_0}v -\lambda v(x,t-\tau )- Mu, &\qtq{in}\RR\times (0,T),\\
v(x,s)=u_0(x,s),&\qtq{in}\RR\times [-\tau, 0].
\end{cases}
\end{equation} 
Note that $A_{\lambda_0}$ generates a strongly continuous semigroup $\{S(t)\}_{t\geq 0}$ of contractions in $L^2(\R)$. Moreover, Proposition \ref{p32} allows us to conclude that (\ref{e9}) has a unique mild solution $v \in B_{0,T}$, such that
\begin{equation}\label{e10}
\begin{split}
\Vert v\Vert_{{\mathcal B}_T}
\le &C_T \left \{ \Vert u_0\Vert_2+\Vert Mu\Vert_{L^1(0,T;L^2(\RR))}\right.\\ &\left.+
  \Vert \lambda\Vert_\infty  \Vert u_0\Vert_{L^1(-\tau, 0; L^2(\RR))} + \Vert \lambda\Vert_\infty^{1/2}  
 \Vert u_0\Vert_{L^2(-\tau, 0; L^2(\RR))}  \right \}
 \end{split}
\end{equation}
where $C_T$ is given by \eqref{36}.
Moreover, we have 
\begin{align*}
\begin{split}
\frac 12 \Vert v(t)\Vert_2^2+\int_0^t\Vert\partial^j_xv\Vert^2_2 ds&+\int_0^t\int_\RR \lambda_0 v^2(x,s)dx ds
+\int_0^t\int_\RR\lambda v(x, s-\tau )v(x, s)dxds \\
&=\frac 12 \Vert u_0\Vert_2^2+\int_0^t\int_\RR Mu(x,s) v(x,s) dxds.
\end{split}
\end{align*}
The idea now is to prove that the operator $\Gamma:\mathcal{B}_{T}\longrightarrow \mathcal{B}_{T}$ is a contraction mapping. To see it, first, note that, thanks to the  Lemma \ref{lm1} and \eqref{e10}, we have
\begin{equation*}
\|\Gamma u \|_{\mathcal{B}_T}\leq C_T \{\|u_0\|_2+  2^{p/2}CT^{\frac{2j-p}{4j}}\|u\|_{\mathcal{B}_T}^{p+1}+ 
  \Vert \lambda\Vert_\infty  \Vert u_0\Vert_{L^1(-\tau, 0; L^2(\RR))} + \Vert \lambda\Vert_\infty^{1/2}  
 \Vert u_0\Vert_{L^2(-\tau, 0; L^2(\RR))} \},
\end{equation*}
and, for $u \in B_R(0):=\{u \in \mathcal{B}_{T}: \|u\|_{\mathcal{B}_{T}}\leq R\}$, it follows that
\begin{equation*}
\|\Gamma u \|_{\mathcal{B}_T}\leq C_T \{\|u_0\|_2+  2^{p/2}CT^{\frac{2j-p}{4j}}R^{p+1}+
  \Vert \lambda\Vert_\infty  \Vert u_0\Vert_{L^1(-\tau, 0; L^2(\RR))} + \Vert \lambda\Vert_\infty^{1/2}  
 \Vert u_0\Vert_{L^2(-\tau, 0; L^2(\RR))} \}.
\end{equation*}
Choosing $$R=2C_T\left(\|u_0\|_2 + \left (
  \Vert \lambda\Vert_\infty  \Vert u_0\Vert_{L^1(-\tau, 0; L^2(\RR))} + \Vert\lambda\Vert_\infty^{1/2}   \Vert u_0\Vert_{L^2(-\tau, 0; L^2(\RR))}  \right )\Vert u_0\Vert_{L^2(-\tau, 0; L^2(\RR))}\right),$$ we obtain the following estimate
\begin{equation*}
\|\Gamma u \|_{0,T}\leq \left(K_1+\frac{1}{2}\right)R,
\end{equation*}
where $K_1=K_1(T)= 2^{p/2}C_TCT^{\frac{2j-p}{4j}}R^{p}.$ On the other hand, note that $\Gamma u - \Gamma w$ is solutions of
\begin{equation*}
\begin{cases}
v_t = A_{\lambda_0}v -\lambda v (t-\tau)- (Mu-Mv), &\qtq{in}\RR\times (0,T),\\
v(x,s)=0,&\qtq{in}\RR\times [-\tau, 0].
\end{cases}
\end{equation*}

We will now prove that $\Gamma$ has a unique fixed point. To do that, note that thanks to Proposition \ref{p32}, we have
\begin{align*}
\|\Gamma u -\Gamma w \|_{\mathcal{B}_T}&\leq C_T\|Mu-Mw\|_{L^1(0,T;L^2)}.
\end{align*}
Lemma \ref{lm2}, precisely estimate (iv),  allows us to conclude that
\begin{align*}
\|\Gamma u -\Gamma w \|_{\mathcal{B}_T}\leq C_TC \left\lbrace 2^{\frac{p}{2}}T^{\frac{2j-p}{4j}}\left(\|u\|_{\mathcal{B}_T}^p+\|u\|_{\mathcal{B}_T}\|w\|_{\mathcal{B}_T}^{p-1} +\|w\|_{\mathcal{B}_T}^{p}\right) +2^{\frac{1}{2}}T^{\frac{1}{4}}\|u\|_{\mathcal{B}_T}\right\rbrace\|u-w\|_{\mathcal{B}_T}.
\end{align*}
Suppose that $u,w \in B_R(0)$ defined above. Then,
\begin{align*}
\|\Gamma u -\Gamma w \|_{\mathcal{B}_T}&\leq K_2 \|u-w\|_{\mathcal{B}_T},
\end{align*}
where $$K_2=K_2(T)= C_TC\{2^{\frac{1}{2}}T^{\frac{1}{4}}R +3(2^{\frac{p}{2}})T^{\frac{2j-p}{4j}}R^{p}\}.$$ Since  $K_1\leq K_2$, we can  choose $T>0$ to obtain $K_2<\frac{1}{2}$,
$$\|\Gamma u\|_{\mathcal{B}_T}\leq R\quad \text{and}\quad \|\Gamma u-\Gamma w\|_{\mathcal{B}_T}< \frac{1}{2}\|u-w\|_{\mathcal{B}_T},$$
for all $u,w \in B_R(0) \subset B_{0,T}$.  Hence, $\Gamma: B_R(0) \rightarrow B_R(0)$ is a contraction and, by Banach fixed point theorem, we obtain a unique $u\in B_R(0)$, such that $\Gamma (u)=u$ and consequently,  the local well-posedness result for $0<T\le \tau$ small enough to the system.  Thus, $u$ is a unique local mild solution to the problem, and the estimate \eqref{r1} holds. 
\end{proof}

 We are now able to present the main result of this subsection.
\begin{Theorem}\label{t33}
Consider $\lambda_0, \lambda\in L^\infty (\RR)$ satisfying the hypotheses of Theorem \ref{t36} (or Theorem \ref{t42}). Then the system \eqref{11} admits a unique global mild solution for every initial datum $u_0\in C([-\tau,0]; L^2(\RR))$ satisfying 
\begin{equation}\label{39}
\begin{split}
\frac 12 \Vert u(0)\Vert_2^2=&\frac 12 \Vert u(t)\Vert_2^2+\int_0^t\Vert u_x\Vert^2_2\,  ds+\int_0^t\int_\RR \lambda_0 u^2(x,s)dxds\\&+\int_0^t\int_\RR\lambda u(x, s-\tau )u(x, s)dxds 
\end{split}
\end{equation}
for all $t\ge 0$. Moreover,  there exists a non-decreasing continuous function $\beta_0 : \R_+ \rightarrow \R_+$,  such that the solution $u$ \eqref{11} satisfies
\begin{equation*}
\Vert u\Vert_{{\mathcal B}_T}\le C_T \left \{ \Vert u(0)\Vert_2+  \Vert \lambda\Vert_\infty  \Vert u_0\Vert_{L^1(-\tau, 0; L^2(\RR))} + \Vert \lambda\Vert_\infty^{1/2}  
 \Vert u_0\Vert_{L^2(-\tau, 0; L^2(\RR))}  \right \}
\end{equation*}
and
\begin{equation*}
\|u\|_{\B_T} \leq \beta_0(\|u_0\|_2)\|u_0\|_{C([-\tau,0]; L^2(\RR))},
\end{equation*}
with $\beta_0= C_T\left(1 + \left ( \Vert \lambda\Vert_\infty \tau^{1/2} + \Vert \lambda\Vert_\infty^{1/2} \right )\right)$  and $C_T$ given by \eqref{36}. 
\end{Theorem}

\begin{proof}To prove the global well-posedness, we need to prove that the norms of the solutions of \eqref{11} remain bounded in the existence time interval. To do that, let us consider the functional \begin{equation}\label{25}
{\mathcal E}(t):= {\mathcal E}(u(t))=\frac 12\int_\RR u^2 (x,t) dx+\frac 12\int_{t-\tau }^t\int_\RR e^{-(t-s)}\vert \lambda (x)\vert u^2 (x,s)\ dx\  ds.
\end{equation}
Taking the time derivative of this function, we have that 
\begin{align*}
\begin{split}
\frac {d{\mathcal E}} {dt} (t)
=&\int_\RR u(t)((-1)^j\partial^{2j}_xu(t)-\lambda_0 u(t)-\lambda u(t-\tau )+u^p(t)u_x(t)) dx+\frac 12\int_\RR \vert \lambda\vert u^2(t) dx
\\&- \frac 12 e^{-\tau}\int_\RR \vert \lambda\vert u^2(t-\tau ) dx -\frac 12 \int_{t-\tau}^t \int_\RR e^{-(t-s)}\vert \lambda\vert u^2(x, s) dx ds.
\end{split}
\end{align*}
Integrating by parts, using the Young inequality and taking into account the hypotheses  \eqref{15},  \eqref{27}, and \eqref{28}, we get that
\begin{align*}
\begin{split}
\frac {d{\mathcal E}} {dt} (t)
\le& -\int_\RR (\partial^j_xu)^2(t) dx -\gamma_0 \int_\RR u^2(t) dx +\frac {e^\tau +1} 2 \int_\RR \vert \lambda(x)\vert  u^2 (t) dx\\
& -\frac 12 \int_{t-\tau}^t \int_\RR e^{-(t-s)}\vert \lambda\vert u^2(x, s) dx ds\\
\le&
-\int_\RR (\partial^j_xu)(t) dx - (\gamma_0 -\gamma ) \int_\RR u^2(t) dx +
\int_\RR \beta (x) u^2 (t) dx.
\end{split}
\end{align*}
Now,  the Hölder inequality ensures that
\begin{equation}\label{212}
\frac {d{\mathcal E}} {dt} (t)\le-\Vert \partial^j_xu(t)\Vert_2^2 -(\gamma_0 -\gamma)\Vert u(t)\Vert_2^2
+\Vert \beta\Vert_q \Vert u\Vert_{2q'}^2
\end{equation}
with $q'=\frac q {q-1}.$ Observing that
\begin{equation}\label{213}
\Vert u\Vert_{2{q'}}^2=\left (\int_\RR (u(t))^{2{q'}} dx\right )^{\frac {1}{q'}}= \left (\int_\RR u^2(t) (u(t))^{\frac 2 {q-1}} dx\right )^{\frac {1}{ q'}}
\le \Vert u\Vert_2^{\frac 2 {q'}}\Vert u\Vert_\infty^{\frac 2 {{q'}(q-1)}}= \Vert u\Vert_2^{\frac 2 {q'}}\Vert u\Vert_\infty^{\frac 2 q}
\end{equation}
and using \eqref{213} in \eqref{212}, yields  
\begin{align}\label{214a}
\begin{split}
\frac {d{\mathcal E}} {dt} (t)\le&-\Vert \partial^j_xu(t)\Vert_2^2 -(\gamma_0 -\gamma )\Vert u(t)\Vert_2^2+\Vert \beta\Vert_q \Vert u(t)\Vert_2^{\frac 2 {q'}} \Vert u(t)\Vert_{\infty}^{\frac 2 q}.
\end{split}
\end{align}
Using \eqref{215} and \eqref{e87} and applying Young's inequality, we deduce from \eqref{214a} for every fixed $\delta >0$ the following estimates:
\begin{equation*}
\begin{split}
\frac {d{\mathcal E}} {dt} (t)
\le&
-\Vert \partial^j_xu(t)\Vert_2^2 -(\gamma_0 -\gamma)\Vert u(t)\Vert_2^2+2^{1/q} 
\Vert \beta \Vert_q\Vert u\Vert_2^{\frac {2q-1} q} \Vert u_x\Vert_2^{\frac 1 q} \\
\le&
-\Vert \partial^j_xu(t)\Vert_2^2 -(\gamma_0 -\gamma)\Vert u(t)\Vert_2^2+2^{\frac 1 q}C^{\frac 1 q} 
\Vert \beta \Vert_q\Vert u\Vert_2^{\frac {2qj-1}{qj}} \Vert \partial_x^j u\Vert_2^{\frac{1}{qj}} \\
\le& -\Vert \partial^j_xu(t)\Vert_2^2 -(\gamma_0 -\gamma)\Vert u(t)\Vert_2^2
+\frac {
\Big ( \frac {1}{\delta } \Vert \beta \Vert_q\Vert u\Vert_2^{\frac {2qj-1}{qj}}  \Big )^\frac {2qj}{2qj-1}}{\frac {2qj}{2qj-1}} + \frac {
 \left ( \delta 2^{\frac 1 q}C^{\frac 1 q}\Vert \partial_x^j u\Vert_{2}^{\frac {1}{qj}}     \right )^{2qj}}{2qj}.
 \end{split}
\end{equation*}
Picking $\delta>0$ such that $2^{2j}C^{2j}\delta^{2qj}=2qj,$ this gives
\begin{align*}
\frac {d{\mathcal E}} {dt} (t)\le& - \Big (
\gamma_0 -\gamma  -\frac {2qj-1}{2qj} \Big (\frac {2^{2j-1}C^{2j}} {qj} \Big )^{\frac 1 {2qj-1}}
\Vert \beta \Vert_q^{\frac {2qj}{2qj-1}}
\Big ) \Vert u(t)\Vert_2^2\leq0.
\end{align*}
This inequality ensures that $\Vert u(t)\Vert_2$ remains bounded for $t\in [0, T].$

Finally,  thanks to the estimate \eqref{39}, we deduce that $\Vert u\Vert_{\mathcal{B}_T}$ remains bounded  for $t\in [0, T]$, and so,  the local solution $u$ given by Proposition \ref{r1} can  be extended on $[0,\tau].$ Now,  once we have a solution $u\in {\mathcal B}_\tau$, we can apply the same arguments as we did on Theorem  \ref{t22} to prove the existence of a global mild solution of \eqref{11}. Finally, note that the proof of  Proposition \ref{r1} guarantees the  function $\beta_0$ is given by 
\begin{align*}
\beta_0(s)= C_T\left(1 + \left ( \Vert \lambda\Vert_\infty \tau^{1/2} + \Vert \lambda\Vert_\infty^{1/2} \right )\right),
\end{align*}
completing the proof.
\end{proof}
\begin{Remark}
The previous results are still valid when $\lambda(x)=0$. All the results of this subsection will be used several times in the next one, considering $\lambda(x)=0$.
\end{Remark}
\subsubsection{\textbf{Well-posedness theory in $H^{2j+1}(\R)$}}We will analyze the well-posedness in $\B_{2j+1,T}$, with $1\leq p <2j$ with $j\geq 1$.  To do that,  let us first consider the following linearized problem:
\begin{equation}\label{ee105}
\begin{cases}
v_t +(-1)^{j+1}\partial_x^{2j+1}v+(-1)^m\partial_x^{2m} v+\partial_x(u^pv)\\\hspace{4cm} +\lambda_0(x) v + \lambda (x) v(x, t-\tau)=0,&\qtq{in}\RR\times (0,\infty),\\
v(x,s)=v_0(x,s),&\qtq{in}\RR\times [-\tau, 0],
\end{cases}
\end{equation}
when  $\lambda (x)\neq0 $ and 
\begin{equation}\label{ee105aa}
\begin{cases}
v_t +(-1)^{j+1}\partial_x^{2j+1}v+(-1)^m\partial_x^{2m} v+\partial_x(u^pv) +\lambda_0(x) v=0,&\qtq{in}\RR\times (0,\infty),\\
v(x,0)=v_0(x),&\qtq{in}\RR,
\end{cases}
\end{equation}
when  $\lambda (x)=0 $.
We can establish the following proposition:

\begin{Proposition}\label{prop4}
For $T>0$, $\lambda_0, \lambda \in L^{\infty}(\R)$, and $u \in \B_{T}$, if $v_0 \in C([-\tau,0];L^2(\R))$, then the system \eqref{ee105} admits a unique solution $v \in \mathcal{B}_{0,T}$, such that 
\begin{equation}\label{106_1}
\Vert v(t)\Vert_
{C([-\tau,T];L^2(\RR))}\le
C(\|\lambda\|_{\infty},T,\tau)
\left ( \Vert v_0\Vert_{C([-\tau,0];L^2(\RR))}+(p+1)\Vert u\Vert_{\B_T}^p\Vert v\Vert_{\B_T}\right ),
\end{equation}
where $C(\|\lambda\|_{\infty},T,\tau)$ is a positive constant and
\begin{equation}\label{106_2}
\Vert v\Vert_{{\mathcal B}_T}\le  C_T \left \{ \Vert v(0)\Vert_2+\Vert u\Vert_{\B_T}^p+
\Vert \lambda\Vert_\infty\Vert v_0\Vert_{L^1(-\tau, 0; L^2(\RR))}    + \Vert \lambda\Vert_\infty^{1/2} \Vert v_0\Vert_{L^2(-\tau, 0; L^2(\RR))}  \right \}
\end{equation}
and
\begin{equation}\label{e106}
\|v\|_{\B_T} \leq \sigma(\|u\|_{\B_T},T)\left \{ \Vert v(0)\Vert_2+
\Vert \lambda\Vert_\infty\Vert v_0\Vert_{L^1(-\tau, 0; L^2(\RR))}    + \Vert \lambda\Vert_\infty^{1/2} \Vert v_0\Vert_{L^2(-\tau, 0; L^2(\RR))}  \right \}.
\end{equation}
Here,  $1\leq p<2j$, with $j\geq 1$, and $\sigma: \R^+ \times \R^+ \rightarrow \R^+$ is a non-decreasing  continuous function. 
\end{Proposition}
\begin{proof} 
The existence of a solution follows the same steps as done in Proposition \ref{r1} and Theorem \ref{t33}, so we omit it.  Let us prove the inequality \eqref{e106}.  Thanks to \eqref{35_1} and \eqref{35}, we deduce \eqref{106_1} and 
\begin{equation*}
\begin{split}
\Vert v\Vert_{{\mathcal B}_T}\le& C_T \left \{ \Vert v(0)\Vert_2+\Vert (u^pv)_x\Vert_{L^1(0,T;L^2(\RR))}\right.\\&\left.+
\Vert \lambda\Vert_\infty\Vert v_0\Vert_{L^1(-\tau, 0; L^2(\RR))}    + \Vert \lambda\Vert_\infty^{1/2} \Vert v_0\Vert_{L^2(-\tau, 0; L^2(\RR))}  \right \}.
\end{split}
\end{equation*}
Lemma \ref{lm2}   and Young's inequality imply that there exists a constant $C>0$ such that
\begin{equation*}
\Vert v(t)\Vert_
{C([-\tau,T];L^2(\RR))}\le
e^{\Vert\lambda\Vert_\infty T}
\left ( \Vert v(0)\Vert_2+(p+1)\Vert u\Vert_{\B_T}^p\Vert v\Vert_{\B_T}\right )
\end{equation*}
and
\begin{equation*}
\Vert v\Vert_{{\mathcal B}_T}\le C C_T \left \{ \Vert v(0)\Vert_2+\Vert u\Vert_{\B_T}^p+
\Vert \lambda\Vert_\infty\Vert v_0\Vert_{L^1(-\tau, 0; L^2(\RR))}    + \Vert \lambda\Vert_\infty^{1/2} \Vert v_0\Vert_{L^2(-\tau, 0; L^2(\RR))}  \right \},
\end{equation*}
giving \eqref{106_2}. On the other hand, by multiplying \eqref{ee105} by $v$, using \eqref{help} and following the same ideas as in the proof of Proposition \ref{p32}, we infer that
\begin{align*}
\begin{split}
\Vert v(t)\Vert_2^2 +2\int_0^t \Vert \partial^j_xv\Vert_2^2ds\le &2( \Vert \lambda\Vert_\infty +\Vert \lambda_0\Vert_\infty)
\int_0^t\Vert v\Vert^2_2ds\\
&+\Big (1+e^{2\Vert \lambda\Vert_\infty T}\Big )
\Big \{ \Vert v(0)\Vert_2+\Vert \partial_x (u^p v) v \Vert_{L^1(0,t;L^2(\RR))}  \\
  &+ 
  \Vert \lambda\Vert_\infty  \Vert v\Vert_{L^1(-\tau, 0; L^2(\RR))} + \Vert \lambda\Vert_\infty^{1/2}  
 \Vert v\Vert_{L^2(-\tau, 0; L^2(\RR))}  \Big \}^2.
\end{split}
\end{align*}
Using the Gagliardo--Nirenberg inequality \eqref{e87} and  integrating by parts, we obtain that
\begin{align*}
\Vert \partial_x (u^p v) v\Vert_{L^1(0,t;L^2(\RR))} &= \Vert (u^{p} v) v_x\Vert_{L^1(0,t;L^2(\RR))}  \leq \int_0^t \| u^{p}(s) v_x(s) v(s)\|_2ds \\
& \leq \|u\|_{\B_T}^p \int_0^t \| v_x(s)\|_{2}\| v(s)\|_{\infty}ds   \leq 2^{\frac12}\|u\|_{\B_T}^p \int_0^t \| v_x(s)\|^{\frac32}_2\| v(s)\|_{2}^{\frac{1}{2}}ds \\
& \leq 2^{\frac12}C\|u\|_{\B_T}^p \int_0^t \| v(s)\|^{\frac32 \left( 1 -\frac{1}{j}\right)}_2\|\partial_x^j v(s)\|^{\frac{3}{2j}}_2 \| v(s)\|_{2}^{\frac{1}{2}}ds  \\
& =2^{\frac12}C\|u\|_{\B_T}^p \int_0^t \| v(s)\|^{\frac{4j-3}{2j} }_2\|\partial_x^j v(s)\|^{\frac{3}{2j}}_2 ds .
\end{align*} 
Furthermore, by Young's inequality, we have
\begin{align*}
\int_0^t \| v(s)\|^{\frac{4j-3}{2j} }_2\|\partial_x^j v(s)\|^{\frac{3}{2j}}_2 ds  \leq \frac{(4j-3)}{4j \eta^{\frac{4j}{4j - 3}}}  \int_0^t \| v(s)\|^2_2ds + \frac{3 \eta^{\frac{4j}{3}}}{4j} \int_0^t\|\partial_x^j v(s)\|^2_2 ds  
\end{align*}
with $\eta=\left[\frac{4j}{3\sqrt{2}C (1+e^{2\|\lambda\|_{\infty}T})\|u\|^p_{\B_T}}\right]^{\frac{3}{4j}}$. Hence, we obtain that
\begin{align*}
\begin{split}
\Vert v(t)\Vert_2^2 +\int_0^t \Vert \partial^j_xv\Vert_2^2ds\le& \rho(\|u\|_{\B_T})
\int_0^t\Vert v(s)\Vert^2_2ds +\Big (1+e^{2\Vert \lambda\Vert_\infty T}\Big )
\Big \{ \Vert v(0)\Vert_2 \\
  &+ 
  \Vert \lambda\Vert_\infty  \Vert v\Vert_{L^1(-\tau, 0; L^2(\RR))} + \Vert \lambda\Vert_\infty^{1/2}  
 \Vert v\Vert_{L^2(-\tau, 0; L^2(\RR))}  \Big \}^2
\end{split}
\end{align*}
with 
\begin{equation*}
\begin{split}
\rho(\|u\|_{\B_T}) =&2( \Vert \lambda\Vert_\infty +\Vert \lambda_0\Vert_\infty)\\&+   (1+e^{2\Vert \lambda\Vert_\infty T})2^{\frac12}C\|u\|_{\B_T}^p \frac{(4j-3)}{4j }\left[\frac{4j}{3\sqrt{2}C (1+e^{2\|\lambda\|_{\infty}T})\|u\|^p_{\B_T}}\right]^{-\frac{3}{4j-3}} \\
=&2( \Vert \lambda\Vert_\infty +\Vert \lambda_0\Vert_\infty)\\&+ (4j-3)\left[ \frac{3\sqrt{2}C (1+e^{2\|\lambda\|_{\infty}T})\|u\|^p_{\B_T}}{4j}\right]^{\frac{4j}{4j-3}}.
\end{split}
\end{equation*}
Finally, we have  
\begin{align*}
\begin{split}
\Vert v(t)\Vert_2^2 +\int_0^t \Vert \partial^j_xv\Vert_2^2ds\le& \rho(\|u\|_{\B_T})
\int_0^t \left( \Vert v(s) \Vert^2_2 +  \int_0^s \Vert \partial^j_xv(t)\Vert_2^2dr \right) ds \\
&+\Big (1+e^{2\Vert \lambda\Vert_\infty T}\Big )
\Big \{ \Vert v(0)\Vert_2+ 
  \Vert \lambda\Vert_\infty  \Vert v\Vert_{L^1(-\tau, 0; L^2(\RR))}  \\
  & + \Vert \lambda\Vert_\infty^{1/2}  
 \Vert v\Vert_{L^2(-\tau, 0; L^2(\RR))}  \Big \}^2.
\end{split}
\end{align*}
Employing Gronwall’s inequality, we conclude that
\begin{align*}
\begin{split}
\Vert v(t)\Vert_2^2 +\int_0^t \Vert \partial^j_xv\Vert_2^2ds\le &\Big (1+e^{2\Vert \lambda\Vert_\infty T}\Big )e^{\rho(\|u\|_{\B_T}) t}
\Big \{ \Vert v(0)\Vert_2+ 
  \Vert \lambda\Vert_\infty  \Vert v\Vert_{L^1(-\tau, 0; L^2(\RR))}  \\
  & + \Vert \lambda\Vert_\infty^{1/2}  
 \Vert v\Vert_{L^2(-\tau, 0; L^2(\RR))}  \Big \}^2.
\end{split}
\end{align*}
Estimate \eqref{e106} follows directly from the above inequality.
\end{proof}

Note that the previous proposition remains valid even when $\lambda(x) = 0$. We will use this fact to establish the existence of a solution to problem~\eqref{11} in the absence of the delay term.
\begin{Theorem}\label{prop3}
Let $T>0$, $\lambda=0$, $ \lambda_0 \in H^{j}(\R)$  and $1\leq p<2j$, with $j\geq 1$.  If
\begin{equation*}
u_0 \in H^{2j+1}(\R),
\end{equation*}
then there exists a unique mild solution $u \in \mathcal{B}_{2j+1,T}$ of \eqref{11} such that
\begin{equation*}
\|u\|_{\B_{2j+1,T}} \leq \beta_{2j+1}(\|u_0\|_2)\|u_0\|_{H^{2j+1}(\R)},
\end{equation*}
where $\beta_{2j+1}: \R_+\rightarrow \R_+$ is a non-decreasing continuous function.
\end{Theorem}

\begin{proof} It is important to keep in mind that in this proof $\lambda=0$. We split the proof into several steps.
\vspace{1mm}

\noindent \textbf{Step 1: $u \in L^2(0,T; H^{2j+1}(\R)).$}

\vspace{1mm}

Since $u_0 \in H^{2j+1}(\R) \hookrightarrow L^2(\R)$, by Theorem \ref{t33} there exists a unique solution $u \in \B_{T}$, such that
\begin{equation}\label{e85}
\|u\|_{\B_T} \leq \beta_0(\|u_0\|_2)\|u_0\|_{L^2(\RR)}.
\end{equation}
We will show that $u \in \B_{2j+1,T}$. Let $v=u_t$. Then, $v$ solves the problem \eqref{ee105aa} with
\begin{multline*}
v_0(x)=u_t(x,0)=-(-1)^{j+1}\partial_x^{2j+1}u_0(x)+(-1)^m\partial_x^{2m} u_0(x)
-\frac{1}{p+1}\partial_x(u_0^{p+1})(x)-\lambda_0(x) u_0(x).
\end{multline*}
We can bound $v_0$ as follows:
\begin{align*}
\|v_0\|_2 &\leq \|\partial_x^{2j+1}u_0\|_2+\|\partial_x^{2m}u_0\|_2+\|u_0^p\partial_xu_0\|_2+\|\lambda_0u_0\|_2   \\
&\leq  (1+\|\lambda_0\|_{L^{\infty}(\R)})\|u_0\|_{H^{2j+1}(\R)}+\|u_0\|^{\frac{p}{2}}_{\infty}\|\partial_x u_0\|_2 \\
&\leq (1+\|\lambda_0\|_{L^{\infty}(\R)})\|u_0\|_{H^{2j+1}(\R)}+2^{\frac{p}{2}}\|u_0\|^{\frac{p}{2}}_{2}\|\partial_x u_0\|^{\frac{p+2}{2}}_2 .
\end{align*}
Using the Gagliardo--Nirenberg inequality, we have
\begin{equation}\label{e87'}
\|\partial_xu_0\|_2 \leq C \|\partial^{2j+1}_xu_0\|^{\frac{1}{2j+1}}_2\|u_0\|_2^{1-\frac{1}{2j+1}}.
\end{equation}
Applying \eqref{e87'}, we ensure that
\begin{equation*}
\begin{split}
\|v_0\|_2\leq&  (1+\|\lambda_0\|_{L^{\infty}(\R)})\|u_0\|_{H^{2j+1}(\R)}\\&+2^{\frac{p}{2}}C\|\partial^{2j+1}_xu_0\|^{\frac{p+2}{2(2j+1)}}_2\|u_0\|_2^{\frac{p+2}{2}\left(1 -\frac{1}{2j+1} \right)+ \frac{p}{2}}.
\end{split}
\end{equation*}
Then Young's inequality guarantees the following:
\begin{equation*}
\|v_0\|_2 \leq (1+\|\lambda_0\|_{L^{\infty}(\R)})\|u_0\|_{H^{2j+1}(\R)}+2^{\frac{p(2j+1)}{4j-p}}\|u_0\|^{\left(\frac{p+2}{2}\left(1 -\frac{1}{2j+1} \right)+ \frac{p-2}{2} \right) \frac{2(2j+1)}{4j-p}}_{2}\|u_0\|_{2}+\|\partial_x^{2j+1} u_0\|_2 
\end{equation*}
This leads to
\begin{equation}\label{e127}
\begin{split}
\|v_0\|_2&\leq C(\|u_0\|_2)\|u_0\|_{H^{2j+1}(\R)}
\end{split}
\end{equation}
with $$C(s)=  2+\|\lambda_0\|_{L^{\infty}(\R)}+2^{\frac{p(2j+1)}{4j-p}}s^{\left(\frac{p+2}{2}\left(1 -\frac{1}{2j+1} \right)+ \frac{p-2}{2} \right) \frac{2(2j+1)}{4j-p}}.$$ Thanks to Proposition \ref{prop4}, we see that $v \in B_{0,T}$ and
\begin{equation*}
\|v\|_{\B_T} \leq  \sigma (\|u\|_{\B_T})\|v_0\|_{L^{2}(\R)}.
\end{equation*}
Combining \eqref{e85} and \eqref{e127}, we get
\begin{equation}\label{e86}
\|v\|_{\B_T} \leq \sigma (\beta_0(\|u_0\|_2)\|u_0\|_2)C(\|u_0\|_2)\|u_0\|_{H^{2j+1}(\R)},
\end{equation}
which means
\begin{equation}\label{e75}
u, u_t \in L^2(0,T;H^j(\R)).
\end{equation}
Therefore,
\begin{equation}\label{e70}
u \in C([0,T];H^j(\R))  \hookrightarrow  C([0,T];C(\R)).
\end{equation}

On the other hand, note that $u^pu_x, \lambda_0 u$ belong to $ L^2(0,T;L^2(\R))$. Moreover, $$(-1)^{j+1}\partial_x^{2j+1}u+(-1)^{m}\partial_x^{2m}u=-u_t-u^pu_x-\lambda_0u\text{ in }D'(0,T,\R).$$ Hence,
\begin{equation*}
(-1)^{j+1}\partial_x^{2j+1}u+(-1)^{m}\partial_x^{2m}u = f \in L^2(0,T;L^2(\R)).
\end{equation*}
Taking the Fourier transform, we have
\begin{equation*}
\widehat{u}= \frac{\widehat{f}+\widehat{u}}{[\left(1-((-1)^{j+1}(i\xi)^{2j+1})-(-1)^{j}(i\xi)^{2j}\right)]}
\end{equation*}
and
\begin{align}\label{e73}
\|u(t)\|^2_{H^{2j+1}(\R)} \leq C \left\lbrace \|f(t)\|_2^2+\|u(t)\|_2^2\right\rbrace
\end{align}
with
\begin{equation*}
C=\sup_{\xi \in \R}\frac{1+|\xi|+|\xi|^{2}+|\xi|^{3}+\cdots+|\xi|^{2j+1}}{\left|1-((-1)^{j+1}(i\xi)^{2j+1})-(-1)^{j}(i\xi)^{2j}\right|}.
\end{equation*}
Integrating \eqref{e73} over $[0,T]$, we deduce that
\begin{equation}\label{e74}
u  \in L^2(0,T;H^{2j+1}(\R)),
\end{equation}
achieving Step 1.

\vspace{1mm}

\noindent \textbf{Step 2: $u \in C(0,T;H^{2j+1}(\mathbb{R}))$}

\vspace{1mm}

Observe that, according to \eqref{e75}, $u_t \in L^2(0,T;H^{-(2j+1)}(\R))$. Then, considering the Gelfand triplet
$H^{2j+1}(\R)\hookrightarrow H^{2m}(\R)\hookrightarrow H^{-(2j+1)}(\R)$,
by \cite[Chapter III - Lemma 1.2]{temam1984theory} we have $u  \in C([0,T];H^{2m}(\R)).$
This implies further
\begin{equation}\label{e78}
\partial_x^{2m}u, \lambda_0 u \in C([0,T];L^2(\R))\cap L^2(0,T;H^1(\R)).
\end{equation}
On the other hand, note that
\begin{equation*}
\begin{split}
\|u(t)^pu_x(t)-u(t_0)^pu_x(t_0)\|_{2} \leq&  \|[u(t)^p-u(t_0)^p]u_x(t)\|_{2}+\|u(t_0)^p[u_x(t)-u_x(t_0)]\|_{2} \\
\leq &C\left\lbrace \|(1+|u(t)|^{p-1}+|u(t_0)|^{p-1})|u(t)-u(t_0)|u_x(t)\|_2 \right.\\
&\left.+\|(1+|u(t_0)|^{p})|u_x(t)-u_x(t_0)|\|_2\right\rbrace\\
\leq &C\left\lbrace (1+\|u(t)\|^{p-1}_{\infty}+\|u(t_0)\|^{p-1}_{\infty})\|u(t)-u(t_0)\|_{\infty}\|u_x(t)\|_2 \right.\\
&\left.+(1+\|u(t_0)\|^{p}_{\infty})\|u_x(t)-u_x(t_0)\|_2\right\rbrace.
\end{split}
\end{equation*}
Then,  the regularity given in \eqref{e70} ensures that
\[
\lim_{t\rightarrow t_0}\|u(t)^pu_x(t)-u(t_0)^pu_x(t_0)\|_{2}=0
\]
and, therefore $u^pu_x \in C([0,T];L^2(\R))$. The above results  also guarantee that
\begin{equation}\label{e82}
u^pu_x \in C([0,T];L^2(\R)) \cap  L^2(0,T;H^1(\R)).
\end{equation}
Indeed, since $(u^pu_x)_x= pu^{p-1}u_x^2 + u^pu_{xx}$, it is sufficient to combine \eqref{e70},  \eqref{e74} and the following two estimates:
\begin{align*}
\|pu^{p-1}u_x^2\|_{L^2(0,T;L^2(\R))}\leq C\left\lbrace (1+\|u\|_{C([0,T];C(\R))}^{p-1})\|u_x\|_{C([0,T];C(\R))}\|u_x\|_{L^2([0,T];L^2(\R))}\right\rbrace
\end{align*}
and 
\begin{align*}
\|u^pu_{xx}\|_{L^2(0,T;L^2(\R))}&\leq C \left\lbrace (1+\|u\|_{C([0,T];C(\R))}^{p})\|u_{xx}\|_{L^2([0,T];L^2(\R))}\right\rbrace
\end{align*}
to ensure \eqref{e82}.

Now, since
\begin{align}\label{eqnew7}
(-1)^{j+1}\partial_x^{2j+1}u &=-(-1)^{m}\partial_x^{2m}u-u_t-(u^pu_x)-\lambda_0u,
\end{align}
from \eqref{e75},  \eqref{e78},  \eqref{e82}, we obtain
\begin{equation}\label{e83}
u \in  L^2(0,T;H^{3j+1}(\R)) \hookrightarrow L^2(0,T;H^{2j+2}(\R))  .
\end{equation}
Finally, considering the Gelfand triplet
\begin{equation*}
H^{2j+2}(\R)\hookrightarrow H^{2j+1}(\R)\hookrightarrow H^{-(2j+2)}(\R),
\end{equation*}
\cite[Chapter III - Lemma 1.2]{temam1984theory} gives
\begin{equation}\label{e84}
u \in  C([0,T];H^{2j+1}(\R)),
\end{equation}
and Step 2 is proved. Note that \eqref{e83} and \eqref{e84} imply that $u \in \B_{2j+1,T}$.

\vspace{1mm}

\noindent \textbf{Step 3:}  The following estimate holds: $$\|u\|_{C([0,T];H^{2j+1}(\R))} \leq \sigma_1(\|u_0\|_2)\|u_0\|_{H^{2j+1}(\R)}.$$

\vspace{1mm}

Indeed,  according to \eqref{e73}  we get
\begin{align}\label{e88}
\|u(t)\|_{H^{2j+1}(\R)} \leq C\left\lbrace \|u_t(t)\|_2+\|u^p(t)u_x(t)\|_2+\|\lambda_0u(t)\|_2 +\|u(t)\|_2 \right\rbrace.
\end{align}
Now, by using the Gagliardo--Nirenberg inequality, we obtain that
$$
\|u(t)^pu_x(t)\|_2 \leq 2^{\frac{1}{p}}\|u(t)\|_2^{\frac{p}{2}}\|u_x(t)\|_2^{\frac{p+2}{2}}
 \leq 2^{\frac{1}{p}} C\|u(t)\|_2^{\frac {4j(p+1)+p}{2(2j+1)}}\|u(t)\|_{H^{2j+1}(\R)}^{\frac{p+2}{2(2j+1)}}.$$
Moreover,  Young's inequality gives
\begin{align*}
\|u(t)^pu_x(t)\|_2 &\leq C\|u(t)\|_2^{  \frac {4j(p+1)+p}{4j-p}}+\frac{1}{2C}\|u(t)\|_{H^{2j+1}(\R)}.
\end{align*}
Substituting the above estimate  into the inequality  \eqref{e88}, we get 
\begin{align*}
\|u(t)\|_{H^{2j+1}(\R))} \leq C\left\lbrace \|u_t\|_{\B_T}+(1+\|\lambda_0\|_{\infty}) \|u\|_{\B_T}+\|u\|_{\B_T}^{ \frac {4j(p+1)+p}{4j-p}}\right\rbrace.
\end{align*}
Then, using \eqref{e85} and \eqref{e86} it follows that
\begin{align*}
\|u(t)\|_{H^{2j+1}(\R))} \leq& C\left\lbrace \sigma (\beta_0(\|u_0\|_2)\|u_0\|_2)C(\|u_0\|_2)\|u_0\|_{H^{2j+1}(\R)}+(C+\|\lambda_0\|_{\infty})\beta_0(\|u_0\|_2)\|u_0\|_{2} \right. \notag \\
&\left.+\beta_0^{ \frac {4j(p+1)+p}{4j-p}} (\|u_0\|_2)\|u_0\|_{2}^{ \frac {4j(p+1)+p}{4j-p}-1} \|u_0\|_{2} \right\rbrace \notag\\
\leq& \bar{\sigma}_1(\|u_0\|_2)\|u_0\|_{H^{2j+1}(\R)}
\end{align*}
with $$\bar{\sigma}_1(s)=C\left\lbrace \sigma (\beta_0(s)s)C(s)+(C+\|b\|_{\infty})\beta_0(s)+(\beta_0(s))^{\frac {4j(p+1)+p}{4j-p}}s^{\frac {4j(p+1)+p}{4j-p}-1}\right\rbrace.$$ 
Therefore, we obtain that
\begin{align*}
\|u\|_{C([0,T];H^{2j+1}(\R)))} \leq   \bar{\sigma}_1(\|u_0\|_2)\|u_0\|_{H^{2j+1}(\R)}.
 \end{align*}
Finally, we deduce from Theorem \ref{t33} and the previous inequality  the relation
\begin{align*}
\|u\|_{C([0,T];H^{2j+1}(\R)))} \leq  \sigma_1(\|u_0\|_2) \|u_0\|_{H^{2j+1}(\R)},
\end{align*}
with $\sigma_1(s)=\bar{\sigma}_1(s)$, finishing Step 3.

\vspace{1mm}

\noindent \textbf{Step 4:} The following estimate holds:
$$\|\partial_x^{3j+1}u\|_{L^2(0,T;L^2(\R))} \leq \sigma_3(\|u_0\|_2)\|u_0\|_{H^{2j+1}(\R)}.$$
Indeed, we deduce from the equation \eqref{eqnew7} that 
\begin{align*}
\int_0^T \|u(t)\|_{H^{3j+1}}^2dt &= \int_0^T\|\partial_x^{3j+1}u(t)\|_{2}^2dt +  \int_0^T\|u(t)\|_{2}^2dt
 \\
&\leq \int_0^T\|u(t)\|_{2}^2dt + \int_0^T\|\partial_x^{2m+j}u(t)\|_{2}^2dt  + \int_0^T\|\partial_x^{j}u_t(t)\|_{2}^2dt \\
& \quad  + \int_0^T\|\partial_x^{j}(u^pu_x)\|_{2}^2dt+ \int_0^T\|\partial_x^{j}(\lambda_0 u(t))\|_{2}^2dt \\
& \leq C T \left( \|u\|_{C([0,T];H^{2j+1})}   + \|u\|_{C([0,T];H^{2m+j})}\right)  + \|u_t\|_{B_{0,T}}\\
& \quad  + \int_0^T\|\partial_x^{j}(u^pu_x)\|_{2}^2dt+ \int_0^T\|\partial_x^{j}(\lambda_0 u(t))\|_{2}^2dt . 
\end{align*}
Note that $ H^{3j+1}(\R) \subset H^{2m+j}(\R) $. Therefore, by Step 3, the Gagliardo--Nirenberg inequality \eqref{e87} and Proposition \ref{prop4}, there exists a function $\sigma_2$ such that 
\begin{equation}\label{eqnew10}
\begin{split}
\int_0^T \|u(t)\|_{H^{3j+1}}^2dt  \leq & \sigma_2\left( \|u_0\|_{2}\right)\|u_0\|_{C([-\tau,0];H^{2j+1})} + \int_0^T\|\partial_x^{j}(u^pu_x)\|_{2}^2dt\\&+ \int_0^T\|\partial_x^{j}(\lambda_0 u(t))\|_{2}^2dt
 \end{split}
\end{equation}
We must estimate the last two integrals on the right-hand side of \eqref{eqnew10}.  
For the estimate of the last integrals, we observe that 
$$\partial_x^{j}( \lambda_0 u)= \sum_{k=0}^j \left(\begin{array}{c}
j \\
k
\end{array} \right) \partial_x^k (\lambda_0) \partial_x^{j-k}(u)$$ 
Using it, we obtain that
\begin{equation}\label{eqnew11}
\begin{split}
\int_0^T\|\partial_x^{j}( \lambda_0 u)(t)\|_2^2dt \leq& \sum_{k=0}^j \int_0^T\left(\begin{array}{c}
j \\
k
\end{array} \right) \|\partial_x^k (\lambda_0) \partial_x^{j-k}(u)(t)\|_2^2dt \\
=&\int_0^T\|\lambda_0\|_{\infty}^2\| \partial_x^{j}(u)(t)\|_2^2dt\\&+\sum_{k=1}^j\int_0^T \left(\begin{array}{c}
j \\
k
\end{array} \right) \|\partial_x^k (\lambda_0) \|_2^2\|\partial_x^{j-k}(u)(t)\|_{\infty}^2 dt  \\
\leq& C \|\lambda_0\|_{H^1(\R)}^2 \int_0^T\| \partial_x^{j}(u)(t)\|_2^2dt\\&+\sum_{k=1}^j \int_0^T\left(\begin{array}{c}
j \\
k
\end{array} \right) \|\lambda_0 \|_{H^k(\R)}^2\|u(t)\|_{H^{1+j-k}(\R)}^2dt  \\
\leq &C T \|\lambda_0 \|_{H^j(\R)}^2 \|u\|_{C([0,T];H^{2j+1}(\R)}^2 .
\end{split}
\end{equation}

Now, we concentrate on the estimate of the first integral on the right-hand side of \eqref{eqnew10}.  Observing that
\begin{align*}
\partial_x^{j}(u^pu_x)= \sum_{k=0}^j \left(\begin{array}{c}
j \\
k
\end{array} \right) \partial_x^k (u^p) \partial_x^{j-k}(u_x),
\end{align*}
we get
\begin{align*}
\int_0^T\|\partial_x^{j}(u^pu_x)(t)\|_{2}^2dt &\leq  \sum_{k=0}^j \left(\begin{array}{c}
j \\
k
\end{array} \right) \int_0^T \|\partial_x^k (u^p)(t) \|_{\infty}^2\|\partial_x^{j-k}(u_x)(t)\|_2^2dt \\
&= \sum_{k=0}^j \left(\begin{array}{c}
j \\
k
\end{array} \right) \int_0^T \|\partial_x^k (u^p)(t) \|_{\infty}^2\|u(t)\|_{H^{1+j-k}(\R)}^2dt \\
&\leq C \sum_{k=0}^j \left(\begin{array}{c}
j \\
k
\end{array} \right) \int_0^T \|\partial_x^k (u^p)(t) \|_{\infty}^2\|u(t)\|_{H^{2j+1}(\R)}^2dt. 
\end{align*}
On the other hand, note that the $k$-order derivative of the function $u^p$ can be written in the following way:
\begin{align*}
\partial_x^k (u^p) =C_{0,p}u^{p-1}\partial_x^ku+\sum_{n=1}^{k-1} C_{n,p} u^{p-n}F_n(u)+C_{k,p}u^{p-k}u_x^k ,
\end{align*}
where the constants $C_{n,p}$ are given by the formula 
$$C_{n,p}=M_n\prod_{i=0}^n (p-i)$$ with $M_n\in \N$, and  $F_n(u)$ is a differential operator involving sums and products of derivatives of $u$ with order less than $n+1$. Thanks to this fact, we have the following estimate:
\begin{align*}
\|\partial_x^k (u^p)\|_{\infty} \leq&\ |C_{0,p}|\|u\|_{\infty}^{p-1}\|\partial_x^ku\|_{\infty}+\sum_{n=1}^{k-1} |C_{n,p}|\| u\|_{\infty}^{p-n}\|F_n(u)\|_{\infty}+|C_{k,p}|\|u\|_{\infty}^{p-k}\|u_x\|_{\infty}^k \\
\leq&\ C\left(|C_{0,p}|\|u\|_{H^{2j+1}(\R)}^{p-1}\|u\|_{H^{k+1}(\R)}+\sum_{n=1}^{k-1} |C_{n,p}|\| u\|_{H^{2j+1}(\R)}^{p-n}\|u\|_{H^{n+2}(\R)}^{m(k)}\right. \\
& \left.+|C_{k,p}|\|u\|_{H^{2j+1}(\R)}^{p-k}\|u\|_{H^{2}(\R)}^k \right),
\end{align*}
where $m(n) \in \N$ and $C>0$ is a constant. Since
\begin{equation*}
H^{2j+1}(\R)\subset H^{k+1}(\R),\quad
H^{2j+1}(\R)\subset H^{n+1}(\R)\ \text{ and }\ 
H^{2j+1}(\R)\subset H^{2}(\R),
\end{equation*}
it follows that
\begin{align*}
\|\partial_x^k (u^p)\|_{\infty} &\leq C\left(|C_{0,p}|\|u\|_{H^{2j+1}(\R)}^{p}+\sum_{n=1}^{k-1} |C_{n,p}|\| u\|_{H^{2j+1}(\R)}^{p-n+m(k)} +|C_{k,p}|\|u\|_{H^{2j+1}(\R)}^{p} \right).
\end{align*}
It  follows that
\begin{equation*}
\begin{split}
\int_0^T\|\partial_x^{j}(u^pu_x)(t)\|_{2}^2dt \leq &C |C_{0,p}|^2 \int_0^T \|u(t)\|_{H^{2j+1}(\R)}^{2(p+1)}dt \\
&+C \sum_{k=0}^j \sum_{n=0}^k \left(\begin{array}{c}
j \\
k
\end{array} \right) |C_{n,p}|^2 \int_0^T \|u(t)\|_{H^{2j+1}(\R)}^{2(p-n+m(k)+1)}dt \\
&+ C\sum_{k=0}^j \left(\begin{array}{c}
j \\
k
\end{array} \right) |C_{k,p}|^2 \int_0^T \|u(t)\|_{H^{2j+1}(\R)}^{2(p+1)}dt. 
\end{split}
\end{equation*}
Finally, we obtain that
\begin{equation}\label{eqnew13}
\begin{split}
\int_0^T\|\partial_x^{j}(u^pu_x)(t)\|_{2}^2dt \leq& C T \left( |C_{0,p}|^2 \|u\|_{C([0,T];H^{2j+1}(\R))}^{2(p+1)}  \right.\\
&+ \sum_{k=0}^j \sum_{n=0}^k \left(\begin{array}{c}
j \\
k
\end{array} \right) |C_{n,p}|^2  \|u(t)\|_{C([0,T];H^{2j+1}(\R))}^{2(p-n+m(k)+1)} \\
&\left. + \sum_{k=0}^j \left(\begin{array}{c}
j \\
k
\end{array} \right) |C_{k,p}|^2 \|u(t)\|_{C([0,T];H^{2j+1}(\R))}^{2(p+1)} \right).
\end{split}
\end{equation}
Step 3 and thus the proof of the theorem is completed by putting the estimates \eqref{eqnew11}  and \eqref{eqnew13} together into \eqref{eqnew10}. 
\end{proof}

\subsection{Interpolation arguments} In this part of the work, we prove the well-posedness of the system \eqref{11}. To do that, let us introduce an interpolation argument due to Tartar \cite{tartar1972interpolation} and adapted by Bona and Scott \cite[Theorem 4.3]{bona1976solutions}.

Let $B_0$ and $B_1$ be two Banach spaces, where $B_1\subset B_0$ with the inclusion map continuous. Consider $f \in B_0$ and define
\begin{align*}
K(f,t)=\inf_{g \in B_1}\left\lbrace \|f-g\|_{B_0}+t\|g\|_{B_1}\right\rbrace,
\end{align*}
for $t \geq 0$.
For $0<\theta <1$ and $1\leq p \leq +\infty$, we introduce the set
\begin{align*}
B_{\theta,p}:=[B_0,B_1]_{\theta,p}=\left\lbrace f\in B_0: \|f\|_{\theta,p}:\left(\int_{0}^{\infty}K(f,t)t^{-\theta p - 1}dt\right)^{\frac{1}{p}}< \infty \right\rbrace,
\end{align*}
with the usual modification for the case $p=\infty$. Then, $B_{\theta,p}$ is a Banach space with norm $\|\cdot\|_{\theta,p}$. Given two pairs $(\theta_1,p_1)$ and $(\theta_2,p_2)$ as above,  $(\theta_1,p_1)\prec(\theta_2,p_2)$ will denote
\begin{equation*}
\left\lbrace \begin{tabular}{l l}
$\theta_1 < \theta_2,$ & or \\
$\theta_1 = \theta_3,$ &  and $p_1>p_2$.
\end{tabular}\right.
\end{equation*}
If $(\theta_1,p_1)\prec(\theta_2,p_2)$, then $B_{\theta_2,p_2}\subset B_{\theta_1,p_1}$ with a continuous inclusion. 
Then, the following result, proved by \cite[Theorem 4.3]{bona1976solutions},  holds:

\begin{Theorem}\label{inter}
Let $B_{0}^j$ and $B_{1}^j$ be Banach spaces such that $B_{1}^j\subset B_{0}^j$ with continuous inclusion mappings, for $j=1,2$. Let $\alpha$ and $q$ lie in the ranges $0 < \alpha < 1$ and $1\leq q\leq \infty$. Suppose that $\A$ is a mapping satisfying:
\begin{enumerate}
\item[(i)] $\A: B_{\alpha, q}^{1} \rightarrow B_0^2$ and, for $f,g \in  B_{\alpha, q}^{1}$,
\begin{equation*}
\|\A f - \A g\|_{B_0^2}\leq C_0\left(\|f\|_{B_{\alpha, q}^{1}}+\|g\|_{B_{\alpha, q}^{1}} \right)\|f-g\|_{B_0^1}.
\end{equation*}
\item[(ii)] $\A: B_1^1 \rightarrow B_1^2$ and, for $h\in  B_1^{1}$,
\begin{equation*}
\|\A h\|_{B_1^2}\leq C_1\left(\|h\|_{B_{\alpha, q}^{1}} \right)\|h\|_{B_1^1},
\end{equation*}
\end{enumerate}
where $C_j : \R^+ \rightarrow \R^+$ are continuous nondecreasing functions, for $j=0,1$. Then, if $(\theta,p)\geq(\alpha,q)$, $\A$ maps $B_{\theta, p}^{1}$ into $B_{\theta, p}^{2}$ and, for $f \in B_{\theta, p}^{1}$, we have
\begin{equation*}
\|\A f\|_{B_{\theta, p}^2}\leq C\left(\|f\|_{B_{\alpha, q}^{1}} \right)\|f\|_{B_{\theta,p}^1},
\end{equation*}
where $C(r)=4C_0(4r)^{1-\theta}C_1(3r)^{\theta}$,  with $r>0$.
\end{Theorem}

It follows from Theorem \ref{t33} that, for each fixed $T >0$, the solution map
\begin{equation}\label{e100}
\A: L^2(\R) \rightarrow \B_{0,T}, \quad \A u_0=u
\end{equation}
is well-defined. Moreover, we have the following result:

\begin{Proposition}\label{prop2}
The solution map \eqref{e100} is locally Lipschitz continuous, that is, there exists a continuous function $\mathcal{C}_0: \R^+\times (0,\infty) \rightarrow \R^+$, nondecreasing in its first variable, such that, for all $u_0, v_0 \in L^2(\R) $, we have
\begin{equation*}
\|\A u_0-\A v_0\|_{0,T} \leq \mathcal{C}_0\left( \|u_0\|_2+\|v_0\|_2,T\right)\|u_0-v_0\|_{L^2(\R)}.
\end{equation*}
\end{Proposition}

\begin{proof}
Let $0 < \theta \leq T$ and $n=\left[ \frac{T}{\theta}\right]$. Theorem  \ref{t33} ensures that
\begin{equation}\label{e102}
\|\A u_0\|_{\B_{0,\theta}}\leq \beta_0(\|u_0\|_2)\|u_0\|_{L^2(\R)}.
\end{equation}
where $\beta_0$ is a constant function given by Theorem \ref{t33} and
\begin{equation*}
\|\A u_0-\A v_0\|_{0, \theta} \leq C_{\theta} \left\lbrace \|u_0-v_0\|_2+\|M(\A u_0)-M(\A v_0)\|_{L^1(0,\theta;L^2(\R))}\right\rbrace,
\end{equation*}
where $C_{\theta}=2e^{\theta\|b\|_{\infty}}$. Moreover, thanks to  Lemma \ref{lm2} we have
\begin{equation*}
\begin{split}
\|\A u_0-\A v_0\|_{\B_{0, \theta}} \leq& C_{\theta} \|u_0-v_0\|_2  + C_{\theta}C\left\lbrace 2^{\frac{p}{2}}\theta^{\frac{2j-p}{4j}}\left( \|\A u_0\|_{\B_{0, \theta}}^p+\|\A u_0\|_{\B_{0, \theta}}\|\A v_0\|_{\B_{0, \theta}}^{p-1}\right. \right. \notag\\
&\left.\left.+\|\A v_0\|_{\B_{0, \theta}}^{p}\right)+2^{\frac{1}{2}}\theta^{\frac{1}{4}}\|\A u_0\|_{\B_{0, \theta}} \right\rbrace\|\A u_0-\A v_0\|_{\mathcal{B}_{0, \theta}}.
\end{split}
\end{equation*}
Now,  inequality \eqref{e102} together with the following estimate: 
\begin{equation*}
\begin{split}
\|\A u_0-\A v_0\|_{0, \theta}  \leq &C_{\theta} \|u_0-v_0\|_2  + C_{\theta}C\left\lbrace 2^{\frac{p}{2}}\theta^{\frac{2j-p}{4j}}\beta_0^p\left( \| u_0\|^p_{L^2(\R)}\right. \right. \notag\\
&\left.\left. +\|\ u_0\|_{L^2(\R)}\|\ v_0\|^{p-1}_{L^2(\R)} +\|\ v_0\|^p_{L^2(\R)}\right) \right. \\\notag
&\left.+2^{\frac{1}{2}} \beta_0\theta^{\frac{1}{4}}\| u_0\|_{L^2(\R)} \right\rbrace\|\A u_0-\A v_0\|_{\mathcal{B}_{0, \theta}}, 
\end{split}
\end{equation*}
yields that
\begin{equation}\label{e104}
\|\A u_0-\A v_0\|_{\B_{0, \theta}} \leq 2C_{T} \|u_0-v_0\|_2,
\end{equation}
if   $\theta$ is chosen small enough.  Observe that setting 
\begin{equation*}
B_{0,[k\theta,(k+1)\theta]}:=C\left([k\theta,(k+1)\theta];L^2(\R)\right) \cap L^2(k\theta,(k+1)\theta;H^j(\R)).
\end{equation*}
with norm $\|\cdot\|_{\B_{0,[k\theta,(k+1)\theta]}}$, in a analogously way as done to \eqref{e104}, we can deduce thanks to estimate \eqref{e102} that
\begin{equation*}
\begin{split}
\|\A u_0-\A v_0\|_{\B_{0,[k\theta,(k+1)\theta]}} 
\leq& C_{\theta} \|u(k\theta)-v(k\theta)\|_2  + C_{\theta}C\beta_0\left\lbrace 2^{\frac{p}{2}}\theta^{\frac{2j-p}{4j}}\beta_0^p\left( \| u_0\|^p_{L^2(\R)}\right. \right. \notag\\
&\left.\left. +\|\ u_0\|_{L^2(\R)}\|\ v_0\|^{p-1}_{L^2(\R)} +\|\ v_0\|^p_{L^2(\R)}\right) \right. \\\notag
&\left.+2^{\frac{1}{2}} \beta_0\theta^{\frac{1}{4}}\| u_0\|_{L^2(\R)} \right\rbrace\|\A u_0-\A v_0\|_{\B_{0,[k\theta,(k+1)\theta]}}.
\end{split}
\end{equation*}
Lastly, we get
\begin{align}\label{e105}
\|\A u_0-\A v_0\|_{\B_{0,[k\theta,(k+1)\theta]} }\leq 2C_{T} \|u(k\theta)-v(k\theta)\|_2, \quad k=0,1,...,n-1.
\end{align}
On the other hand, note that \eqref{e104} and \eqref{e105} imply that
\begin{align*}
\|\A u_0-\A v_0\|_{\B_{0,[k\theta,(k+1)\theta]}} \leq 2^{k}C_{T}^{k} \|u_0-v_0\|_2, \quad k=0,1,...,n-1,
\end{align*}
and, therefore,
\begin{equation*}
\|\A u_0-\A v_0\|_{\B_{0,[k\theta,(k+1)\theta]}} \leq 2^{n}C_{T}^{n} \|u_0-v_0\|_2.
\end{equation*}
Finally,
\begin{align*}
\|\A u_0-\A v_0\|_{\B_{0,T}}&\leq \sum_{k=0}^{n-1} \|\A u_0-\A v_0\|_{\B_{0,[k\theta,(k+1)\theta]}} \leq \sum_{k=0}^{n-1} 2^{n}C_{T}^{n} \|u_0-v_0\|_2  \\
&\leq  2^{n}C_{T}^{n}n \|u_0-v_0\|_2 \leq \mathcal{C}_0(\|u_0\|_2+\|v_0\|_2)\|u_0-v_0\|_{L^2(\R)}
\end{align*}
with $\mathcal{C}_0(s)=\frac{T}{\theta(s)}\left[2C_T\right]^{\frac{T}{\theta(s)}}$, completing the proof of the proposition.
\end{proof}

We are in a position to prove the main result of this section.

\begin{proof}[Proof of Theorem \ref{teo2}]
We define
\begin{equation*}
B_0^1=L^2(\R),\quad B_0^2=\B_{0,T},\quad B_1^1=H^{2j+1}(\R)\quad \text{and}\quad  B_1^2=\B_{2j+1,T}.
\end{equation*}
Thus,
\begin{equation*}
B_{\frac{s}{2j+1},2}^1=[L^2(\R),H^{2j+1}(\R)]_{\frac{s}{2j+1},2}=H^s(\R) \quad \text{and}\quad  B_{\frac{s}{2j+1},2}^2=[B_{0,T},B_{2j+1,T}]_{\frac{s}{2j+1},2}=B_{s,T}.
\end{equation*}
Combining Proposition \ref{prop2} and Theorem \ref{prop3} we obtain $(i)$ and $(ii)$ in  Theorem \ref{inter}. Then, Theorem \ref{inter} gives the existence of the solution to the equation \eqref{11},  and Theorem \ref{teo2} follows.
\end{proof}

Finally, we prove Corollary \ref{corollarywellposedness}, which establishes that every mild solution of system \eqref{11}, with $\lambda(x)=0$, is indeed a regular solution when the origin is not considered.

\begin{proof}[Proof of Corollary \ref{corollarywellposedness}] The result is shown by using the bootstrap argument. Consider $T>0$ and $0<\varepsilon<T$. So, for $u_0 \in L^2(\RR)$, it follows from Theorem \ref{teo2} that the problem \eqref{11}, with $\lambda(x)=0$,  has a unique solution $$u\in{\mathcal{B}}_{0,T}
= C([0,T];L^2(\RR ))\cap L^2(0,T;H^{j}(\RR)).$$ So, we have  $u(t) \in H^j(\RR)$ for almost every $t \in [0,T]$. Let $t_0 \in (0,\varepsilon)$ such that $u(t_0) \in H^j(\RR)$. Applying Theorem \ref{teo2} and Remark \ref{remark1} with $u_0=u(t_0)$, we conclude that the restriction of $u$ to $[t_0,T]$ is the solution of \eqref{11} with the initial value $u(t_0)$ and $\lambda(x)=0$, and it belongs to the following class:  $${\mathcal{B}}_{j,[t_0,T]}
= C([t_0,T];H^j(\RR ))\cap L^2(t_0,T;H^{2j}(\RR)).$$ 
Thus we have  $u(t) \in H^{2j}(\RR)$ for almost every $t \in [t_0,T]$. Let $t_1 \in (t_0,\varepsilon)$ be such that $u(t_1) \in H^{2j}(\RR)$. Again, it follows from Theorem \ref{teo2} and Remark \ref{remark1} that the restriction of $u$ to $[t_1,T]$ is the solution of \eqref{11} with the initial value $u(t_1)$ and $\lambda(x)=0$, and it belongs to  $${\mathcal{B}}_{2j,[t_1,T]}
= C([t_1,T];H^{2j}(\RR ))\cap L^2(t_1,T;H^{3j}(\RR)).$$ Finally, since $u(t) \in H^{3j}(\RR)$ for almost every $t \in [t_1,T]$ and $H^{3j}(\RR)\subset H^{2j+1}(\RR)$, it follows the same way that for $t_2 \in (t_1,\varepsilon)$  such that $u(t_2) \in H^{2j+1}(\RR)$, the restriction of $u$ to $[t_2,T]$ is the solution of \eqref{11} with the initial value $u(t_2)$ and $\lambda(x)=0$, belonging to the following set:  
$${\mathcal{B}}_{2j+1,[t_2,T]}
:= C([t_2,T];H^{2j+1}(\RR ))\cap L^2(t_2,T;H^{3j+1}(\RR)).$$ 
This completes the proof of the corollary.
\end{proof}

\section{Exponential stabilization: damped and delayed systems}\label{sectionexponentialL2}
This section is devoted to the proof of the exponential stability of the system \eqref{12} in $L^{2}(\R)$. 
First, we prove the result for the linear system, and then we extend it to the nonlinear one. 
To do this, we consider the following Lyapunov functional:
\begin{equation}\label{24}
E(t):= E(u(t))=\frac 12\int_\RR u^2 (x,t) dx,
\end{equation}
and we recall for $\lambda\in L^\infty (\RR )$  the definition of ${\mathcal E}(\cdot)$, given in \eqref{25}.

\subsection{Linear system}
The first result ensures the exponential stability of the linear system \eqref{12} with localized damping and delay terms:

\begin{Theorem}\label{t23}
Let $\lambda , \lambda_0\in L^\infty (\RR)$ and $\lambda_0$  satisfy \eqref{15}. 
If there exist a  constant $\gamma>0$ and a function $\beta\in L^p(\RR),$ for $1\le p<\infty$,  such that the function $\lambda$ satisfies \eqref{27} and \eqref{28}, then the system \eqref{12} is exponentially stable. 
More precisely,  the solutions $u$ of \eqref{12} satisfy the   inequality
\begin{equation}\label{29}
{\mathcal E}(t) \le C(u_0) e^{-\nu_j t},
\end{equation}
where $\nu_j$ and $C(u_0)$ are defined by \eqref{210} and \eqref{211}, respectively.
\end{Theorem}

\begin{proof}
Using the same argument as  in  \cite[Theorem 4.7]{CCKR}, we consider $u_0\in C([-\tau,0];H^{2j+1}(\R));$ then $ u\in C([0,T];H^{2j+1}(\R))$. Taking the derivative in $t$ of ${\mathcal E}(t)$ we get
\begin{equation*}
\begin{split}
\frac {d{\mathcal E}} {dt} (t)
=&\int_\RR u(t)(-(-1)^j\partial^{2j}_xu(t)-\lambda_0 u(t)-\lambda u(t-\tau )) dx
+\frac 12\int_\RR \vert \lambda\vert u^2(t) dx\\
&- \frac 12e^{-\tau}\int_\RR \vert \lambda\vert u^2(t-\tau ) dx
-\frac 12\int_{t-\tau }^t\int_\RR e^{-(t-s)}\vert \lambda (x)\vert u^2 (x,s)dx \, ds,
\end{split}
\end{equation*}
since we have
\begin{equation*}
\int_{\mathbb{R}}(-1)^{j+1}(\partial_x^{2j+1}u)u dx=0.
\qtq{for all}u\in H^{2j+1}.
\end{equation*}
Integrating by parts,  and using the Young inequality,  remembering that  \eqref{15} and \eqref{27} are satisfied, we have
\begin{align*}
\frac {d{\mathcal E}} {dt} (t)
\le &-\int_\RR (\partial^j_xu)^2(t) dx -\gamma_0 \int_\RR u^2(t) dx +\frac {e^\tau +1} 2\int_\RR \vert \lambda(x)\vert  u^2 (t) dx\\
& -\frac 12\int_{t-\tau }^t\int_\RR e^{-(t-s)}\vert \lambda (x)\vert u^2 (x,s)dx ds\\
\le&-\int_\RR (\partial^j_xu)^2(t)  dx - (\gamma_0 -\gamma) \int_\RR u^2(t) dx +
\int_\RR \beta(x) u^2 (t) dx\\
& -\frac 12\int_{t-\tau }^t\int_\RR e^{-(t-s)}\vert \lambda (x)\vert u^2 (x,s)\ dx \ ds;
\end{align*}
we have used the relations $$(-1)^j\int_{\mathbb{R}}u\partial^{2j}_xudx=\int_{\mathbb{R}}(\partial^j_xu)^2dx.$$
Now the Hölder inequality ensures that
\begin{equation}\label{212a}
\frac {d{\mathcal E}} {dt} (t)\le-\Vert \partial^j_xu(t)\Vert_2^2 -(\gamma_0 -\gamma)\Vert u(t)\Vert_2^2
+\Vert \beta\Vert_q \Vert u\Vert_{2q'}^2 -\frac 12\int_{t-\tau }^t\int_\RR e^{-(t-s)}\vert \lambda (x)\vert u^2 (x,s)dx \, ds
\end{equation}
with $q'=\frac q {q-1}.$ Hence, using \eqref{213} in \eqref{212} in \eqref{212a} we obtain that
\begin{align}\label{214}
\begin{split}
\frac {d{\mathcal E}} {dt} (t)\le&-\Vert \partial^j_xu(t)\Vert_2^2 -(\gamma_0 -\gamma )\Vert u(t)\Vert_2^2+\Vert \beta\Vert_q \Vert u(t)\Vert_2^{\frac 2 {q'}} \Vert u(t)\Vert_{\infty}^{\frac 2 q}  \\
&-\frac 12\int_{t-\tau }^t\int_\RR e^{-(t-s)}\vert \lambda (x)\vert u^2 (x,s)dx \, ds.
\end{split}
\end{align}
Now, following the same steps as in the proof of Theorem \ref{t33}, that is, using Young's inequality, \eqref{214}, \eqref{215} and \eqref{e87} in \eqref{214}, we obtain for every fixed $\delta >0$ the following relation:
\begin{equation*}
\begin{split}
\frac {d{\mathcal E}} {dt} (t)
\le& -\Vert \partial^j_xu(t)\Vert_2^2 -(\gamma_0 -\gamma)\Vert u(t)\Vert_2^2
+\frac {
\Big ( \frac {1}{\delta } \Vert \beta \Vert_q\Vert u\Vert_2^{\frac {2qj-1}{qj}}  \Big )^\frac {2qj}{2qj-1}}{\frac {2qj}{2qj-1}} + \frac {
 \left ( \delta 2^{\frac 1 q}C^{\frac 1 q}\Vert \partial_x^j u\Vert_{2}^{\frac {1}{qj}}     \right )^{2qj}}{2qj}\\
& -\frac 12\int_{t-\tau }^t\int_\RR e^{-(t-s)}\vert \lambda (x)\vert u^2 (x,s)dx ds.
\end{split}
\end{equation*}
Picking $\delta>0$ such that $2^{2j}C^{2j}\delta^{2qj}=2qj,$ this yields 
\begin{multline*}
\frac {d{\mathcal E}} {dt} (t)\le - \Big (
\gamma_0 -\gamma  -\frac {2qj-1}{2qj} \Big (\frac {2^{2j-1}C^{2j}} {qj} \Big )^{\frac 1 {2qj-1}}
\Vert \beta \Vert_q^{\frac {2qj}{2qj-1}}
\Big ) \Vert u(t)\Vert_2^2\\
-\frac 12\int_{t-\tau }^t\int_\RR e^{-(t-s)}\vert \lambda (x)\vert u^2 (x,s)dxds.
\end{multline*}
Taking \eqref{28} into account, we infer the inequality
\begin{equation*}
\frac {d{\mathcal E}} {dt} (t)\le -  {\nu_j} {\mathcal E}(t)
\end{equation*}
where $\nu_j$ is as in \eqref{210}. 
We complete the proof by observing that the estimate \eqref{29} is a direct consequence of  Gronwall's Lemma with $C(u_0):={\mathcal E}(0).$
\end{proof}

\subsection{Nonlinear system} In this section, we are interested in proving that the higher-order dispersive system in an unbounded domain is asymptotically stable when we introduce a localized damping mechanism and a delay term.  
We will use the Lyapunov approach to prove that the energy ${\mathcal E}(t)$ defined by \eqref{25} tends to $0$ as $t$ goes to $\infty$.

Observe that we can consider the functions $\lambda_0 , \lambda\in L^\infty (\RR)$  satisfying \eqref{15}, \eqref{27} and \eqref{28}.  Thus, in this case, we have a simple situation that can be easily proved using the arguments developed in the previous section. 

\begin{proof}[Proof of Theorem \ref{t36}] The proof of this result is a consequence of the  Lyapunov functional defined in  \eqref{24} and is analogous to what was done in Theorem \ref{t23}, so we will omit it.
\end{proof}

\subsection{Indefinite damping case} In this section, our issue is to see what happens with general dissipative damping. To be precise, consider the coefficient $\lambda_0$ 
changing sign.  Let us assume that there exist a number $\gamma >0$  and a function $\beta\in L^p(\RR)$ for some $1\le p< \infty$, such that \eqref{41} and \eqref{42} are satisfied. 
We can prove the following asymptotic result for the solutions of the linearized system associated with \eqref{11}, for $m=j$:

\begin{Theorem}\label{t41}
Consider $\lambda , \lambda_0\in L^\infty (\RR)$, with $\lambda_0$ satisfying \eqref{41} and \eqref{42}. If there exist a constant $\gamma>0$ and a function $\beta\in L^p(\RR),$ with the same $p$ as in \eqref{42}, such that the function $\lambda$ satisfies
\eqref{27} and \eqref{43}, then the system 
\begin{equation}\label{12a}
\begin{cases}
u_t(x,t)+(-1)^{j+1}\partial_x^{2j+1}u(x,t)+(-1)^j\partial_x^{2j}u(x,t) +\lambda_0 u(x,t)\\\hspace{4cm}+\lambda u(x,t-\tau )=0
&\qtq{in}\RR\times (0,\infty),\\
u(x, s)=u_0(x, s)&\qtq{in}\RR\times [-\tau, 0]
\end{cases}
\end{equation}
is exponentially stable.  Moreover,  the solution of \eqref{12a} satisfies the estimate
\begin{equation}\label{44}
{\mathcal E}(t) \le C(u_0,\tau) e^{-\tilde\nu_j t},
\end{equation}
with $\tilde\nu_j$ defined by \eqref{45}.
\end{Theorem}

\begin{proof}Taking the derivative in time of ${\mathcal E}(t),$ using the equation \eqref{12a}, and then integrating by parts and using the Young inequality, we get
\begin{align*}
\frac {d{\mathcal E}} {dt} (t)
\le &-\int_\RR (\partial^j_xu)^2(t)\, dx -\gamma_0 \int_\RR u^2(t)\, dx  +\int_\RR (\beta_0 (x)+\frac {e^\tau+1} 2 \vert \lambda (x)\vert ) u^2 (t)\, dx\\
& -\frac 12\int_{t-\tau }^t\int_\RR e^{-(t-s)}\vert \lambda (x)\vert u^2 (x,s)\, dx\,  ds\\
\le&-\int_\RR (\partial^j_xu)^2(t)\,   dx - (\gamma_0 -\gamma) \int_\RR u^2(t)\,  dx +
\int_\RR (\beta_0 (x)+\beta(x) ) u^2 (t)\,  dx\\
& -\frac 12\int_{t-\tau }^t\int_\RR e^{-(t-s)}\vert \lambda (x)\vert u^2 (x,s)\, dx \, ds,
\end{align*}
Hölder's inequality ensures that
\begin{equation}\label{46}
\begin{split}
\frac {d{\mathcal E}} {dt} (t)\le&-\Vert \partial_x^ju(t)\Vert_2^2 -(\gamma_0 -\gamma)\Vert u(t)\Vert_2^2
\\&+\Vert \beta_0 +\beta\Vert_q \Vert u\Vert_{2q'}^2 -\frac 12 \int_{t-\tau}^t\int_\RR e^{-(t-s)} \vert \lambda\vert u^2(x,s)\,  dx \, ds
\end{split}
\end{equation}
with $q'=\frac q {q-1}.$ Thanks to the inequality \eqref{46} and the Gagliardo--Nirenberg inequality \eqref{215}, we deduce that
\begin{align*}
\frac {d{\mathcal E}} {dt} (t)\le&-\Vert \partial^j_xu(t)\Vert_2^2 -(\gamma_0 -\gamma)\Vert u(t)\Vert_2^2\\
&+\Vert \beta +\beta_0\Vert_q  \Vert u(t)\Vert_{\infty}^{\frac 2 q}  \Vert u\Vert_2^{\frac{2}{q'}}-\frac 12 \int_{t-\tau}^t\int_\RR e^{-(t-s)} \vert \lambda\vert u^2(x,s)\,  dx\,  ds \\
\le&-\Vert \partial^j_xu(t)\Vert_2^2 -(\gamma_0 -\gamma)\Vert u(t)\Vert_2^2+2^{\frac 1q}C^{\frac 1q}\Vert \beta +\beta_0\Vert_q  \Vert u(t)\Vert_{2}^{\frac{2qj-1}{qj}  \Vert\partial_x^j u\Vert_2^{\frac{1}{qj}}}\\
&-\frac 12 \int_{t-\tau}^t\int_\RR e^{-(t-s)} \vert \lambda\vert u^2(x,s)\,  dx\,  ds.
\end{align*}
By Young's inequality, we infer that
\begin{align*}
\frac {d{\mathcal E}} {dt} (t)
\le& -\Vert \partial^j_xu(t)\Vert_2^2 -(\gamma_0 -\gamma)\Vert u(t)\Vert_2^2
+\frac {\Big ( \frac {1}{\delta } \Vert \beta+\beta_0 \Vert_q\Vert u(t)\Vert_2^{\frac {2qj-1}{qj}}  \Big )^\frac {2qj}{2qj-1}}{\frac {2qj}{2qj-1}} + \frac {
 \left ( \delta 2^{\frac 1 q}C^{\frac 1 q}\Vert \partial_x^j u(t)\Vert_{2}^{\frac {1}{qj}}     \right )^{2qj}}{2qj}\\
& -\frac 12 \int_{t-\tau}^t\int_\RR e^{-(t-s)} \vert \lambda\vert u^2(x,s)\,  dx\,  ds 
\end{align*}
for every fixed $\delta >0$.
Now taking $\delta>0$ such that $2^{2j}C^{2j}\delta^{2qj}=2qj,$ we obtain that
\begin{align}\label{47}
\begin{split}
\frac {d{\mathcal E}} {dt} (t)\le& - \Big (
\gamma -\gamma_0 \frac {2qj-1}{2qj} \Big (\frac {2^{2j-1}C^{2j}} {qj} \Big )^{\frac 1 {2qj-1}}
\Vert \beta_0 +\beta\Vert_p^{\frac {2p}{2p-1}}
\Big ) \Vert u(t)\Vert_2^2\\
&-\frac 12 \int_{t-\tau}^t\int_\RR e^{-(t-s)} \vert \lambda\vert u^2(x,s)\,  dx\,  ds.
\end{split}
\end{align}
Finally,  under the assumption \eqref{43} and taking \eqref{47} into consideration, we get
\begin{equation*}
\frac {d{\mathcal E}} {dt} (t)\le - \tilde\nu_j {\mathcal E}(t)
\end{equation*}
with
\begin{equation*}
\tilde\nu_j= 2 \Big (
\gamma_0 -\gamma-\frac {2qj-1}{2qj} \Big (\frac {2^{2j-1}C^{2j}} {qj} \Big )^{\frac 1 {2qj-1}}
\Vert \beta +\beta_0\Vert_p^{\frac {2p}{2p-1}}
\Big ).
\end{equation*}
This gives us the exponential estimate \eqref{44} for the solution of \eqref{12a}, with $C(u_0,\tau)$ defined as in \eqref{211}.  
\end{proof}

We can directly extend the well-posedness and stability results to the nonlinear system. Precisely, Theorem \ref{t42} is a consequence of the previous theorem and the results presented in Subsection \ref{s3}. So, with this in hand, let us prove the Corollary \ref{cor2}.

\begin{proof}[Proof of Corollary \ref{cor2}]
Note that, after a change of variable, the restriction of $u$ to $[t,t+T]$ is a solution to the problem \eqref{11} with the initial value $u(t)$. Observe also that $u \in C([\tau_1,t];L^2(\R))$ for all $\tau_1 \in [-\tau, t)$. Thus, by Theorem \ref{t33} we have
\begin{align*}
\|u\|_{\B_{0,[t,t+T]}} &\leq C_{t+T} \left \{ \Vert u(t)\Vert_2+
  \Vert \lambda\Vert_\infty  \Vert u\Vert_{L^1(t-\tau, t; L^2(\RR))} + \Vert \lambda\Vert_\infty^{1/2}  
 \Vert u\Vert_{L^2(t-\tau, t; L^2(\RR))}  \right \},
\end{align*}
with $C_s$ is given by 
\begin{equation*}
C_{s}=\sqrt{\frac 3 2} \left (1+e^{2\Vert \lambda\Vert_\infty s}\right )^{1/2}
e^{(\Vert \lambda\Vert_\infty +\Vert \lambda_0\Vert_\infty) )s}.
\end{equation*}
Now it follows from  Theorem \ref{t42}  that 
\begin{align*}
\|u\|_{\B_{0,[t,t+T]}} 
&\leq C_{2T} \left \{ 2C(u_0,\tau)e^{- \nu t}+
  \Vert \lambda\Vert_\infty  \Vert u\Vert_{L^1(t-\tau, t; L^2(\RR))} + \Vert \lambda\Vert_\infty^{1/2}  
 \Vert u\Vert_{L^2(t-\tau, t; L^2(\RR))}  \right \}
\end{align*}
where $C(u_0,\tau)$ is given by \eqref{211}, completing the proof of the corollary. 
\end{proof}

\section{Exponential stabilization: Damped system}\label{sectionexponentialHsj+1}

In this section, we establish the exponential stability in the space $H^{s}(\R) $ for $s\in [0,2j+1]$, for the general dispersive system \eqref{11} without the time delay term ($\tau=\lambda=0$):\begin{equation}\label{sindelay}
\begin{cases}
u_t(x,t)+(-1)^{j+1}\partial_x^{2j+1}u(x,t)+(-1)^m\partial_x^{2m}u(x,t) +\lambda_0(x) u(x,t)\\\hspace{4cm}+ \frac{1}{p+1}\partial_xu^{p+1}(x,t)=0,
&\qtq{in}\RR\times (0,\infty),\\
 u(x, 0)=u_0(x),&\qtq{in}\RR, 
\end{cases}
\end{equation}
 with  $m\leq j$, $j,m\in\mathbb{N},$ and $ 1\leq p < 2j$. 
 
\begin{Remarks}\label{rmkL}
It is important to point out the following: 
\begin{enumerate}[\upshape (i)]
 \item The strategy to obtain the stabilization results in $H^s(\mathbb{R})$, for any $s\in[0,2j+1]$, will be to prove the result first in $L^2(\mathbb{R})$,  and secondly in the domain of the operator, that is, in $H^{2j+1}(\mathbb{R})$. 
With these two results in hand, we employ the interpolation results due to J.-L. Lions \cite{lions} to get the exponential decay for any $s\in[0,2j+1]$.
 \item In case $\tau=\lambda=0$ the stabilization for the system \eqref{sindelay} in  $L^2(\mathbb{R})$ holds thanks to the Theorem \ref{t36} with localized damping $\lambda_0$ (see Section \ref{sectionexponentialL2}).
\end{enumerate}
\end{Remarks}
 
\subsection{Stabilization in $H^{2j+1}(\RR)$} 
Due to the Remarks \ref{rmkL}, we will consider the stabilization problem associated with the solutions of \eqref{sindelay} in the space $H^{2j+1}(\R)$. 
\begin{Proposition}\label{prop7}
Let $T>0$. For $1\leq p < 2j$, with $j\geq 1$ and $\lambda_0$  satisfying \eqref{15} or \eqref{41},  there exist $\gamma>0$, $T_0>0$ and a nonnegative continuous  function $\alpha_3: \R^+ \rightarrow \R^+$ such that, for every $u_0 \in H^{2j+1}(\R)$, the corresponding solution $u$ of \eqref{sindelay} satisfies
\begin{equation}\label{e129}
\|u(t)\|_{H^{2j+1}(\R)} \leq \alpha_{2j+1}(\|u_0\|_2,T_0)\|u_0\|_{H^{2j+1}(\R)}e^{-\gamma t}, \qquad \forall t\geq T_0.
\end{equation}
\end{Proposition}
\begin{proof}
First, we note that there exists a positive constant $c$ such that the following estimate holds:
$$
\frac{1}{c}\|u(t)\|_{H^{2j+1}(\R)} \leq \|u(t)\|_2 + \|\partial_x^{2j+1}u(t)\|_2,
$$
Then it follows from Theorem \ref{t42} that
\begin{equation}\label{aaaa}
\frac{1}{c}\|u(t)\|_{H^{2j+1}(\R)} \leq 2C(u_0,\tau) e^{-\tilde\nu t} + \|\partial_x^{2j+1}u(t)\|_2, \quad \text{for all $t>0$},
 \end{equation}
with $\tilde\nu$ defined by  \eqref{45}  and $C(u_0,\tau)>0$ as in \eqref{211}. Since $\tau=0$, then $C(u_0,0)=\frac12 \|u_0\|_2^2$. The inequality \eqref{aaaa} shows that we need to establish an exponential estimate for the $(2j+1)$s derivative in the space of $u(t)$. To do that, observe that 
$$
\|\partial_x^{2j+1}u(t)\|_2 \leq \|u_t(t)\| + \|\partial_x^{2m}u(t)\|_2+\|u^pu_x(t)\|_2+\|\lambda_0u(t)\|_2,  \forall t>0. 	
$$
Using the Gagliardo--Nirenberg inequality \eqref{e87} we obtain that
\begin{align*}
\|\partial_x^{2j+1}u(t)\|_2 \leq \|u_t(t)\| + \|\partial_x^{2j+1}u(t)\|_2^{\frac{2m}{2j+1}}\|u(t)\|_2^{1-\frac{2m}{2j+1}}+\|u(t)\|_{\infty}^p\|u_x(t)\|_2
+\|\lambda_0\|_{\infty}\|u(t)\|_2.
\end{align*}
Young's inequality together with Theorem \ref{t42} for $\lambda=\tau=0$ imply that
\begin{align*}
\frac{1}{2}\|\partial_x^{2j+1}u(t)\|_2 &\leq \|u_t(t)\| + C\|u(t)\|_2+2^{\frac{p}{2}}\|u\|^{\frac{p}{2}}_{2}\|\partial_x u\|^{\frac{p+2}{2}}_2 +\|\lambda_0\|_{\infty}\|u(t)\|_2 \\
&\leq \|u_t(t)\| + \left( C +\|\lambda_0\|_{\infty}\right) \|u(t)\|_2+2^{\frac{p}{2}}C\|\partial^{2j+1}_xu\|^{\frac{p+2}{2(2j+1)}}_2\|u\|_2^{\frac{p+2}{2}\left(1 -\frac{1}{2j+1} \right)+ \frac{p}{2}}.   
\end{align*}
Hence, we obtain
\begin{align}\label{eqnew111}
\frac{1}{4}\|\partial_x^{2j+1}u(t)\|_2 &\leq \|u_t(t)\| + \left( C +\|\lambda_0\|_{\infty}\right) \|u(t)\|_2+C\|u\|_2^{\left(\frac{p+2}{2}\left(1 -\frac{1}{2j+1} \right)+ \frac{p}{2}\right)\left( \frac{2(2j+1)}{4j-p}\right)}  \notag \\
&\leq \|u_t(t)\| + 2^{\frac12}C(u_0,0)^{\frac12}\left( C +\|\lambda_0\|_{\infty}\right)  e^{-\frac{\tilde\nu}{2} t}+ 2^{\frac12}C(u_0,0)^{\frac12}e^{-\frac{\tilde\nu}{2} \left( \frac{4j(p+1)+p}{4j-p}\right) t}.   
\end{align}
Let $v=u_t$. Then, by Proposition \ref{prop4}, with $\lambda=\tau=0$, $v$ solves the linearized equation \eqref{ee105} with initial value
\begin{equation*}
v_0(x)=-(-1)^{j+1}\partial_x^{2j+1}u_0+(-1)^m\partial_x^{2m} u_0-\frac{1}{p+1}\partial_x(u_0^{p+1})-\lambda_0(x) u_0 
\end{equation*}
such that 
\begin{equation}\label{e122}
\|v\|_{\B_T} \leq \sigma(\|u\|_{\B_T},T)\Vert v(0)\Vert_2.
\end{equation}
Now, after a change of variable, the restriction of $v$ to $[t, t + T]$ is a solution of the system \eqref{ee105}, considering $\lambda=\tau=0$, concerning the initial value $v(t)$, thus
\begin{equation}\label{e125}
\|v\|_{\B_{[t, t + T]}} \leq \sigma(\|u\|_{\B_{[t, t + T]}},T) \Vert v(t)\Vert_2.
\end{equation}
Applying Corollary \ref{cor2}, it follows that
\begin{align*}
\|v\|_{\B_{[t,t+T]}} &\leq \sigma \Big(  C_{2T} \left \{ 2C(u_0,0)e^{- \nu t}\right \},T \Big)  \Vert v(t)\Vert_2. 
\end{align*}
Therefore, we obtain that
\begin{align*}
\|v\|_{\B_{[t,t+T]}} &\leq \gamma (u_0,t,T)\Vert v(t)\Vert_2,
\end{align*}
where $\gamma(s,t, T)=\sigma \left( 2C_{2T} C(s,0) ),T\right)$.  
On the other hand, the solution $v$ may be written as
\begin{equation*}
v(t)=S(t)v_0-\int_0^t S(t-s)[u^p(s)v(s)]_xds
\end{equation*}
where $S(t)$ is a $C_0$-semigroup of contraction in $L^2(\R)$ generated by the operator associated to system \eqref{ee105}. Note that $v_1(t)=S(t)v_0$ is a solution of the problem (\ref{ee105}) with $u^p=0$. Then, proceeding as in the proof of Theorem \ref{t41} with $\lambda=\tau=0$, we have
\begin{equation}\label{e123}
\|v_1(t)\|_2 \leq 2^{1/2} C(v_0,0)^{1/2}e^{- \frac{\nu}{2} t}, \qquad \forall t \geq 0,
\end{equation}
with $\nu$ defined by \eqref{45}.
Set $$v_2(t)=\int_0^t S(t-s)[u^p(s)v(s)]_xds. $$ Note that
\begin{align*}
 \|v_2(T)\|_2 &\leq \|pu^{p-1}u_xv\|_{L^1(0,T;L^2(\R))}+ \|u^pv_x\|_{L^1(0,T;L^2(\R))}.
\end{align*}
Hence, thanks to  Lemma \ref{lm2}, the following holds:
\begin{align}\label{e124}
\|v_2(T)\|_2 &\leq 2^{\frac{p}{2}}CT^{\frac{2j-p}{4j}}T^{\frac12\left(1-\frac{1}{j}\right)}\|u\|_{\mathcal{B}_{0,T}}^{p}\|v\|_{\mathcal{B}_{0,T}}.
\end{align}
Using \eqref{e122},  \eqref{e123} and \eqref{e124}, we obtain 
\begin{align*}
\|v(T)\|_2 &\leq  \|v_1(T)\|_2+\|v_1(T)\|_2 \leq 2^{1/2} C(v_0,0)^{1/2}e^{- \frac{\nu}{2} t} +2^{\frac{p}{2}}CT^{\frac{2j-p}{4j}}T^{\frac12\left(1-\frac{1}{j}\right)}\|u\|_{\mathcal{B}_{0,T}}^{p}\|v\|_{\mathcal{B}_{0,T}}.
\end{align*}

Let us now consider a sequence $y_n(\cdot)=v(\cdot, nT)$ and set $w_n(\cdot,t)=v(\cdot,t+nT)$. For $t\in [0,T]$, $w_n$ solves the problem
\begin{equation*}
\left\lbrace\begin{tabular}{l l}
$\partial_t w_n+(-1)^{j+1}\partial_x^{2j+1} w_n +(-1)^m\partial_x^{2m} w_n+[u(\cdot+nT)^pw_n]_x+\lambda_0w_n = 0,$ & in $\R\times \R^+,$ \\
$w_n(0)=y_n,$ & in $\R$.
\end{tabular}\right.
\end{equation*}
 Observe that we can obtain for $y_n$ an estimate similar to the one obtained for $v(T)$, namely
\begin{align*}
\|y_{n+1}\|_2&=\|w_{n}(T)\|_2 \leq \|S(T)y_n\|_2 +\left \| \int_0^T S(T-s)[[u(s+nT)^pw_n(s)]_xds \right\|_2 \\
&\leq 2^{1/2} C(v(\cdot+nT),0 )^{1/2}e^{- \frac{\nu}{2} T} +2^{\frac{p}{2}}CT^{\frac{2j-p}{4j}}T^{\frac12\left(1-\frac{1}{j}\right)}\|u(\cdot+nT)\|_{\mathcal{B}_{0,T}}^{p}\|w_n\|_{\mathcal{B}_{0,T}}  \\
&\leq 2^{1/2} C(v(\cdot+nT),0 )^{1/2}e^{- \frac{\nu}{2} T} + 2^{\frac{p}{2}}CT^{\frac{2j-p}{4j}}T^{\frac12\left(1-\frac{1}{j}\right)}\|u\|_{\mathcal{B}_{0,[nT,(n+1)T]}}^{p}\|v\|_{\mathcal{B}_{0,[nT,(n+1)T]}}.
\end{align*}
Thus, \eqref{e125} implies that
\begin{multline*}
\|y_{n+1}\|_2\leq 2^{1/2} C(v(\cdot+nT),0 )^{1/2}e^{- \frac{\nu}{2} T}\\   
+2^{\frac{p}{2}}CT^{\frac{2j-p}{4j}}T^{\frac12\left(1-\frac{1}{j}\right)}\|u\|_{\mathcal{B}_{0,[nT,(n+1)T]}}^{p}\sigma(\|u\|_{\B_{[nT,(n+1)T]}},T) \|y_n\|_2
\end{multline*}
with 
\begin{equation*}
C(v(\cdot+nT),0)=\frac 12\Vert y_n \Vert^2_2,
\end{equation*}
whence
\begin{align}\label{yn}
\|y_{n+1}\|_2 \leq  \Big( e^{- \frac{\nu}{2}  T} + 2^{\frac{p}{2}}CT^{\frac{2j-p}{4j}}T^{\frac12\left(1-\frac{1}{j}\right)}\|u\|_{\mathcal{B}_{0,[nT,(n+1)T]}}^{p}\sigma(\|u\|_{\B_{[nT,(n+1)T]}},T)\Big)\|y_n\|_2.
\end{align}
Moreover, we can choose $\beta >0$ small enough, such that
\begin{equation*}
e^{-\frac{\nu}{2}  T} +2^{\frac{p}{2}}CT^{\frac{2j-p}{4j}}T^{\frac12\left(1-\frac{1}{j}\right)}\beta^{p}\beta\sigma(\beta,T) < 1.
\end{equation*}
With this choice of $\beta$, Corollary \ref{cor2} allows us to pick $N>0$, large enough, satisfying
\begin{equation*}
\|u\|_{0,[nT,(n+1)T]}\leq 2 C_{2T}C(u_0,0) e^{-\nu nT} \leq 2 C_{2T}C(u_0,0) e^{-\nu NT} \leq  \beta, \qquad \forall n> N.
\end{equation*}
Thus, from \eqref{yn} we obtain the following estimate
\begin{equation*}
\|y_{n+1}\|_2\leq r \|y_n\|_2, \qquad \forall n\geq N, \quad \text{where $0<r<1$},
\end{equation*}
which implies
\begin{equation}\label{e126}
\|v((n+k)T)\|_2\leq r^k \|v(nT)\|_2, \qquad \forall n\geq N.
\end{equation}
Now, pick $T_0=NT$ and $t \geq T_0$. Then, there exists $k \in \N$ and $\theta \in [0,T]$, satisfying
\begin{equation*}
t=(N+k)T+\theta.
\end{equation*}
Therefore, from \eqref{e125} and \eqref{e126},  it is follows that
\begin{align*}
\|v(t)\|_2& \leq \|v\|_{0,[(N+k)T, (N+k+1)T]} \leq \gamma (u_0,t,T)\|v((N+k)T)\|_2 \\
&\leq \gamma (u_0,t,T)r^{\frac{t-NT-\theta}{T}}\|v(T_0)\|_2 \\
&\leq \gamma (u_0,t,T)r^{\frac{t-NT-\theta}{T}}\sigma ((\|u\|_{\mathcal{ B}_{T_0}},T_0))\|v(0)\|_2 \\
&\leq \eta_1(\|u_0\|)e^{-\delta_1t}\|v_0\|_2,
\end{align*}
where $\delta_1=\frac{1}{T}\ln\left(\frac{1}{r}\right)$ and $\eta_1(s)=\gamma (u_0,t,T)\sigma(\beta_0(s,T_0))r^{-(N+1)}$. 
Hence
\begin{equation*}
\|u_t(t)\|_2\leq C\eta_1(\|u_0\|_2)\|u_0\|_{H^{2j+1}(\R)}e^{-\delta_1t}, \qquad \forall t \geq T_0.
\end{equation*}
Finally,  from \eqref{eqnew111} and since $C(u_0,0)^{1/2}\leq C\|u_0\|_{H^{2j+1}}$,  we have
\begin{equation*}
\begin{split}
\frac{1}{4}\|\partial_x^{2j+1}u(t)\|_2 \leq& C\eta_1(\|u_0\|_2)\|u_0\|_{H^{2j+1}(\R)}e^{-\delta_1t} \\
&+\left( C +\|\lambda_0\|_{\infty}\right)  e^{-\frac{\tilde\nu}{2} t}\|u_0\|_{H^{2j+1}}+ e^{-\frac{\tilde\nu }{2}\left( \frac{4j(p+1)+p}{4j-p}\right) t} \|u_0\|_{H^{2j+1}},
\end{split}
\end{equation*}
showing the inequality \eqref{e129}.
\end{proof}

\subsection{Stabilization in  $H^{s}(\R)$}
Finally, we prove the third main result of this work.  

\begin{proof}[Proof of Theorem \ref{teoexpHs}] 
We already know that there exists a unique solution $u$ of the system \eqref{11} in the class $\mathcal{B}_{s,T}$ for every $T>0$. Moreover, if $0<\varepsilon<T$ then it follows from Corollary \ref{corollarywellposedness}, $u\in \mathcal{B}_{2j+1,[\varepsilon,T]}$. On the other hand, from the interpolation inequality \cite[inequality (2.43)]{lions}, we have
\begin{equation*}
\|u(t)\|_{H^s(\R)}=\|u(t)\|_{[L^2(\R),H^{2j+1}(\R)]_{2,\frac{s}{2j+1}}} \leq C\|u(t)\|_2^{1-\frac{s}{2j+1}}\|u(t)\|_{H^{2j+1}(\R)}^{\frac{s}{2j+1}}, \quad \forall t\geq \varepsilon,
\end{equation*}
where $C>0$ is a constant that comes from the interpolation argument. Thus,  Proposition \ref{prop7} and Theorems \ref{t36} and \ref{t42}, with $\lambda=\tau=0$, imply the existence of $T_0>0$, $\nu>0$ and $\gamma>0$ such that
\begin{equation*}
\|u(t)\|_{H^s(\R)}\leq Ce^{-\left(1-\frac{s}{2j+1}\right)\nu t}\|u_0\|_2^{\left(1-\frac{s}{2j+1}\right)}\alpha_{2j+1}^{\frac{s}{2j+1}}(\|u(\varepsilon)\|_2,T_0)\|u(\varepsilon)\|^{\frac{s}{2j+1}}_{H^{2j+1}(\R)}e^{-\frac{s}{2j+1}\gamma t}, \qquad \forall t \geq T_0.
\end{equation*}
Hence, \eqref{eqnew118} holds with $$\eta(s)=\left(1-\frac{s}{2j+1}\right)\nu+ \frac{s}{2j+1}\gamma >0$$ and $$\gamma(s,T_0)=Cs^{-\frac{s}{2j+1}}\alpha_3^{\frac{s}{2j+1}}(\|u(\varepsilon)\|_2,T_0)\|u(\varepsilon)\|^{\frac{s}{2j+1}}_{H^{2j+1}},$$
and the proof is complete.
\end{proof}

\section{Concluding remarks}\label{sectionconclusion}
In this work, we gave a more general framework to treat stabilization problems for a general higher-order dispersive system, which extends several previous results that appear in the literature.  So, in terms of generality,  we can consider the nonlinear general differential operator 
\begin{equation}\label{vr}
\mathcal{V}u:=(-1)^{j+1}\partial_x^{2j+1}u(x,t)+(-1)^m\partial_x^{2m}u(x,t)+ \frac{1}{p+1}\partial_xu^{p+1}(x,t),
\end{equation} 
with $1\leq p < 2j$,  instead of the typical KdV equation (or co-related systems) as is usual in the literature\footnote{One may generalize the linear operator associated to $\mathcal{V}$ as
\[\mathcal{V}u= \sum_{m=0}^j \alpha_m \px^{2m+1}u+ \sum_{m=0}^j \beta_m \px^{2m}u,\]
where $\alpha_m,\beta_m \in \mathbb{R}$. However, the main analyses in the paper are almost analogous without additional difficulties, thus the operator \eqref{vr} does not lose the generality in the sense of the aim of this paper.}. Thus,  summarizing, we studied the asymptotic behavior of the equation \eqref{11} posed on an unbounded domain $\mathbb{R}$ with the constant $\tau >0$ as a time delay, and with the coefficients $$\lambda_0(x), \lambda(x) \in L^\infty(\RR) \quad \text{and} \quad j,m\in\mathbb{N}.$$

Considering $p=1$,  when we have $j=m=1$, in \eqref{vr}, we recover the result proved in \cite{KomPig2020} for the KdV-Burgers operator.  Additionally,  when $j=2$ and $m=1$, we have that the results of this manuscript are still valid for the fifth-order KdV-Burgers type operator.  Finally, we can take,  without loss of generality,  $j=m \in \mathbb{N}$ in \eqref{vr}. So,  we can define Lyapunov functionals associated with the solution of \eqref{11}
\begin{equation*}
E(t):= E(u(t))=\frac 12\int_\RR u^2 (x,t) dx
\end{equation*}
and, for $\lambda\in L^\infty (\RR ),$
\begin{equation*}
{\mathcal E}(t):= {\mathcal E}(u(t))=\frac 12\int_\RR u^2 (x,t) dx+\frac 12\int_{t-\tau }^t\int_\RR e^{-(t-s)}\vert \lambda (x)\vert u^2 (x,s)\ dx\  ds.
\end{equation*}
 Thanks to the damping mechanism and the delay term, we showed that solutions of \eqref{11} satisfy
\begin{equation*}
{\mathcal E}(t) \le C(u_0) e^{-\nu t}.
\end{equation*}

Additionally to that, it is important to note that Section \ref{sectionexponentialHsj+1} extends to 
the operator \eqref{vr} and the space $H^s$, for $s\in[0,2j+1]$, the results proposed in \cite{CCKR,gallego}. 
It also extends for an unbounded domain the results shown in \cite{CaVi,Valein}  with an appropriate choice of $j$ in the operator \eqref{vr}.  However,  considering the full system \eqref{11}, that is, the system with damping and delayed terms, is still an open problem to prove the stabilization and well-posedness results in $H^{2j+1}(\mathbb{R})$. Finally,  let us give some further comments. 

\subsection{Weak versus strong damping mechanism}
Observe that taking $\beta(x)=0$ in  Theorems \ref{t36} and \ref{t42}, the results are still valid. However,  we keep this term for a more general framework.  Keeping in mind that $\beta(x)\neq0$, additionally, we can consider $\lambda_0$ and $\lambda$ are constants such that $\vert\lambda\vert<\lambda_0,$ and the delay $\tau$ is sufficiently small,  so Theorem \ref{t36} gives us the exponential stability for the solution of \eqref{11}. This is possible using the method introduced in \cite{NPSicon06} for wave equations. In fact, in this case, we choose a sufficiently small delay $\tau$ and a constant $\gamma$ such that
\begin{equation*}
\frac {e^\tau+1}2 \vert \lambda\vert <\gamma<\lambda_0.
\end{equation*}
Since now $\lambda_0=\gamma_0$ (see \eqref{15}), the conditions \eqref{27} and \eqref{28} are satisfied, and the Theorem \ref{t36}  follows.  
The same remark applies if $\lambda, \lambda_0\in L^\infty (\RR)$ with $\lambda_0$ satisfying \eqref{15}, and $\vert\lambda\vert <\gamma_0.$ 

With respect to the Theorem \ref{t42}, for a more general framework, if $\lambda_0$ satisfies \eqref{41} and \eqref{42} instead of $\lambda_0 (x) \ge \gamma_0$ for a.e. $x\in \RR$, then for the same function, $\beta$ satisfying
$$\frac {e^\tau +1} 2\vert \lambda (x)\vert \le \gamma+ \beta(x)\qtq{for a.e.}x\in\RR,$$
we expect a smaller decay rate $\tilde \nu$  than $\nu$ in Theorem \ref{t42}. The explanation for this is the fact that there is a ``good'' part $\gamma$ of $\lambda_0$ that will compensate for the delayed feedback and its indefinite component $\beta_0$.

\subsection{General framework} One may generalize the system \eqref{11} as follows:
\begin{equation}\label{general_1}
\begin{cases}
\pt u(x,t) + \sum\limits_{m=0}^j \alpha_m \px^{2m+1}u(x,t)+ \sum\limits_{k=1}^n\beta_k \px^{2k}u (x,t)+\lambda_0(x) u(x,t)\\\hspace{4cm}+\lambda(x) u(x,t-\tau )+\frac{1}{2}\partial_xu^2(x,t)=0,
&\qtq{in}\RR\times (0,\infty),\\
u(x, s)=u_0(x, s),&\qtq{in}\RR\times [-\tau, 0],
\end{cases}
\end{equation}
where $\alpha_m,\beta_k \in \mathbb{R}$ and $j,n\in\mathbb{N}$. The previous system, which depends on the parameters $\alpha_m$ and $\beta_k$, recovers various delayed dispersive equations. Note that depending on the choice of the constants, we have:
\begin{itemize}
\item[1.] Burgers equation ($\alpha_m=0$ and $n=1$);
\item[2.] KdV equation ($j=1$ and $\beta_k=0$);
\item[3.] Kawahara equation ($j=2$,  $\alpha_0=1$, $\alpha_1=-1$ and $\beta_k=0$);
\item[4.] KdV--Burgers equation ($j=1$ and $n=1$);
\item[5.] Kawahara--Burgers equation ($j=2$,  $\alpha_0=1$, $\alpha_1=-1$ and $n=1$);
\item[6.] Fourth-order dispersive equation ($\alpha_m=0$, $n=2$, $\beta_1=-1$ and $\beta_2=1$).
\end{itemize}
However,  we point out that for the system \eqref{general_1}, the main analyses in the paper are almost analogous without additional difficulties, thus the equation \eqref{11} does not lose the generality in the sense of the aim of this paper.

\subsection*{Acknowledgments} We greatly appreciate the referee’s careful reading and helpful suggestions. This work was done during several visits of the authors to the Université de Strasbourg and Universidade Federal de Pernambuco.  They thank the members of the departments for their hospitality. Moreover, the first author would like to thank the Department of Mathematics of Virginia Tech, where the article was finished.

\subsection*{Funding} Capistrano–Filho was supported by CNPq grants 307808/2021-1, 401003/2022-1, and 200386/2022-0, CAPES grants 88881.311964/2018-01 and 88881.520205/2020-01, and MATHAMSUD 21-MATH-03.  Gallego was supported by MATHAMSUD 21-MATH-03 and the UnalManizales project Nro 57774.  Komornik was supported by the grant NSFC No. 11871348 and MATHAMSUD 21-MATH-03. 

\subsection*{Data availability statement} The present manuscript has no associated data.
\subsection*{Conflict of interest} The authors declare no conflict of interest.

\end{document}